\documentclass{cedram-alco}
\usepackage{breakurl}	
\usepackage{url}
\usepackage{xcolor}
\usepackage[ps,dvips,color,all]{xy}
\usepackage{upgreek}
\usepackage{pstricks}
% ...

\title
[Combinatorial, piecewise-linear, and birational homomesy]
{Combinatorial, piecewise-linear, and birational homomesy for products of two chains}

\author[\initial{D.} Einstein]{\firstname{David} \lastname{Einstein}}

\address{University of Massachusetts Lowell\\ Department of Mathematical Sciences}

\email{deinst@gmail.com}

\author[\initial{J.} Propp]{\firstname{James} \lastname{Propp}}
\address{University of Massachusetts Lowell\\ Department of Mathematical Sciences}

\urladdr{https://jamespropp.org}

\thanks{The authors were partially supported by NSF Grant \#1001905.}



\newcommand{\cB}{{\mathcal{B}}}

\newcommand{\cF}{{\mathcal{F}}}

\newcommand{\cI}{{\mathcal{I}}}
\newcommand{\cK}{{\mathbb{K}}}
\newcommand{\cO}{{\mathcal{O}}}
\newcommand{\cP}{{\mathcal{P}}}

\newcommand{\row}{\rho}
\newcommand{\rowP}{\rho_{\cP}}
\newcommand{\rowH}{\overline{\rho}_{\cP}}
\newcommand{\rowB}{\rho_{\cB}}
\newcommand{\rowM}{\overline{\rho}_{\cB}}
\newcommand{\rowF}{\rho_{\cF}}
\newcommand{\pro}{\pi}
\newcommand{\proP}{\pi_{\cP}}
\newcommand{\proH}{\overline{\pi}_{\cP}}
\newcommand{\proB}{\pi_{\cB}}
\newcommand{\proM}{\overline{\pi}_{\cB}}

\newcommand{\C}{\mathbb{C}}
\newcommand{\N}{\mathbb{N}}

\newcommand{\R}{\mathbb{R}}
\newcommand{\Z}{\mathbb{Z}}
\newcommand{\sB}{\upalpha}
\newcommand{\sT}{\upomega}
\newcommand{\Sup}{S_{\rm up}}
\newcommand{\Sdown}{S_{\rm down}}
\newcommand{\psum}{{\:\ \mathclap{\|}\mathclap{-}\ \:}}
\newcommand{\covers}{\gtrdot}
\newcommand{\coveredby}{\lessdot}
\newcommand{\symmdiff}{\bigtriangleup}
\newcommand{\bigsig}{{\sum}}
\newcommand{\polytope}{K}
\newcommand{\mR}{\stackrel{\rho_\cB}{\mapsto}}
\newcommand{\mP}{\stackrel{\pi_\cB}{\mapsto}}
\newcommand{\nR}{\stackrel{\rho_\cP}{\mapsto}}
\newcommand{\nP}{\stackrel{\pi_\cP}{\mapsto}}
\newcommand{\recom}{\mathfrak{R}}
\newcommand{\decom}{\mathfrak{D}}

\newcommand{\bigsum}{\bigsig^+}
\newcommand{\bigpsum}{\bigsig^{\ \mathclap{\|}\mathclap{-}}\:}
\newcommand{\nh}{\ar@{-}[ru]}
\newcommand{\sh}{\ar@{-}[rd]}
\newcommand{\jim}{\!\!\!\mathfrak{R}\:\rotatebox[origin=c]{270}{$\mapsto$}}

\newcommand{\rFM}{\stackrel{\rowF}{\mapsto}}

\begin{document}
\begin{abstract}
This article illustrates the dynamical concept of {\em homomesy} in 
three kinds of dynamical systems -- combinatorial, piecewise-linear, and birational --
and shows the relationship between these three settings.  In particular, we show
how the rowmotion and promotion operations of Striker and Williams~\cite{SW12}
can be lifted to (continuous) piecewise-linear operations on the order polytope 
of Stanley~\cite{Sta86}, and then lifted to birational operations on the positive 
orthant in $\R^{|P|}$ and indeed to a dense subset of $\C^{|P|}$.
When the poset $P$ is a product of a chain of length $a$ and a chain of length $b$,
these lifted operations have order $a+b$, and exhibit the homomesy phenomenon:
the time-averages of various quantities are the same in all orbits.
One important tool is a concrete realization of the conjugacy between rowmotion 
and promotion found by Striker and Williams; this {\em recombination map}
allows us to use homomesy for promotion to deduce homomesy for rowmotion.
\end{abstract}

\maketitle

\begin{section}{Introduction}
\label{sec:intro}

Many authors \cite{BS74,CF95,Fon93,Pan09,SW12}
have studied an operation $\row$ on the set of order ideals of a poset $P$
that, following Striker and Williams, we call {\em rowmotion}.
In exploring the properties of rowmotion,
Striker and Williams also introduced and studied 
a closely related operation $\pro$ they call {\em promotion} 
on account of its ties with promotion of Young tableaux,
which depends on the choice of an {\em rc embedding}
(a particular kind of embedding of $P$ into the poset $\Z \times \Z$ 
that sharpens the idea of a Hasse diagram).
In this article (an expanded version of a 2014 FPSAC presentation~\cite{EP14})
we mostly focus on a very particular case,
where $P$ is of the form $[a] \times [b]$ 
and the rc embedding sends $(i,j) \in P$ to $(j-i,i+j-2) \in \Z^2$ 
(the standard Hasse embedding; see Figure~\ref{fig:rc}),
and we explore how the cardinality of an order ideal $I$ 
behaves as one iterates rowmotion and promotion.
Indeed, we find regularities for the average cardinality
of sets of the form $I \cap S_{\ell}$
as $I$ varies over the elements of a rowmotion-orbit or promotion-orbit,
where $\{S_1,S_2,\dots,S_{a+b-1}\}$ is a partition of $[a] \times [b]$
into special sets called {\em files}
(which Striker and Williams call columns).

Let $\cI(P)$ denote the set of order ideals of a poset $P$
(usually written as $J(P)$ in the literature).
It has long been known~\cite{BS74}
that the order of $\pro$ or $\row$
acting on $\cI([a] \times [b])$ is $a+b$.
Propp and Roby~\cite{PR13} showed 
that the average of $|I|$
as $I$ varies over an orbit in $\cI([a] \times [b])$ is $ab/2$,
and sketched a proof of a more detailed claim:

\begin{theo}
\label{thm:cardrefined}
Fix $a,b \geq 1$, let $n=a+b$, let $P = [a] \times [b]$,
and for $1 \leq \ell \leq n-1$
let $S_{\ell} = \{ (i,j) \in P \ | \ j-i+a=\ell \}$.
Then for every order ideal $I$ in $\cI(P)$,
$$\frac{1}{n} \sum_{k=0}^{n-1} |\pro^k (I) \cap S_{\ell}| = 
  \frac{1}{n} \sum_{k=0}^{n-1} |\row^k (I) \cap S_{\ell}| = 
\left\{ \begin{array}{ll}
b\ell/n & \mbox{if $\ell \leq a$}, \\ 
a(n-\ell)/n & \mbox{if $\ell \geq a$}.
\end{array} \right.$$
Summing over $\ell$, we obtain
$$\frac{1}{n} \sum_{k=0}^{n-1} |\pro^k (I)| = 
  \frac{1}{n} \sum_{k=0}^{n-1} |\row^k (I)| = ab/2.$$
\end{theo}
\noindent
(Here as elsewhere in the article, overlap between cases is intentional; 
it is easily checked that the answers given in borderline cases are consistent.)

It is no accident that the same averages are seen
for the promotion operation $\pi$ and the rowmotion operation $\row$;
the recombination principle discussed in Section~\ref{sec:recombine}
explains why we get same averages for both actions.
In some cases we will only state our results for promotion,
but in every case considered here (specifically, in 
Theorems~\ref{thm:sumrefined},~\ref{thm:sumhrefined},~\ref{thm:prodhrefined},
and~\ref{thm:prodrefined}) one may replace promotion by rowmotion 
without changing the common value of the orbit-averages.

The notion of looking at the average of a quantity over an orbit
was an outgrowth of the second author's work 
on chip-firing and rotor-routing~\cite{H+08,HP10};
see in particular Proposition 3 of~\cite{PR13}.
Further inspiration came from conjectures of Panyushev~\cite{Pan09}
(later proved by Armstrong, Stump, and Thomas~\cite{AST11}).

This article presents a new proof of Theorem~\ref{thm:cardrefined} 
(see Section~\ref{sec:tropic}) which,
although less direct than the Propp-Roby proof, 
indicates that the constant-averages-over-orbits phenomenon
(also called the homomesy phenomenon)
applies not just for actions on order ideals
but also for dynamical systems of a different sort. 
Specifically, we define (continuous) piecewise-linear maps
from the order polytope of $P$ to itself 
({\em piecewise-linear rowmotion} and {\em promotion}) 
that exhibit homomesy,
and birational maps from a dense open subset of $\C^{ab}$ to itself
({\em birational rowmotion} and {\em promotion}) 
that exhibit a multiplicative version of homomesy.
(See Subsection~\ref{subsec:homomesy} for definitions of these terms.)

Our main result is Theorem \ref{thm:prodrefined},
whose proof involves three main ingredients.
The first is the main result of Grinberg and Roby (Theorem 30 in~\cite{GR15}),
showing that birational rowmotion on the poset $[a] \times [b]$ has period $a+b$.
The second ingredient is the recombination operation 
explained in Section~\ref{sec:recombine}.
Recombination equivariantly takes birational rowmotion to birational promotion,
giving a concrete way of chopping up rowmotion-orbits 
and reassembling the pieces to obtain promotion-orbits, or vice versa;
the recombination picture tells us that 
birational promotion has the same orbit structure
(and hence the same period) as birational rowmotion.
Recombination has its roots in the work of Striker and Williams
(see Theorem 5.4 in~\cite{SW12}),
but the more detailed combinatorial picture presented here is required 
if we want to prove not just results about periodicity
but also results about homomesy.  
The third ingredient is Lemma~\ref{lem:swap} 
in which the specific nature of promotion
and the specific structure of the poset $[a] \times [b]$ play crucial roles.
The lemma concretizes and exploits the intuition that,
viewed from the correct perspective, promotion can be seen as a form of rotation,
but in a different manner than in the work of Grinberg-Roby.


The plan of the article is as follows.
In Section~\ref{sec:background}, after introducing needed preliminaries and notation,
including the definition of (additive) homomesy,
we review some of the background on the rowmotion and promotion operations 
$\row,\pro: \cI(P) \rightarrow \cI(P)$.
We then define (in Section~\ref{sec:PL})
piecewise-linear maps $\rowP,\proP: \R^{|P|} \rightarrow \R^{|P|}$
and show that $\rowP$ and $\proP$ specialize to $\row$ and $\pro$ 
if one restricts attention to the vertices of the order polytope $\cO(P)$
(replacing order ideals by filters as required).
Changing variables, we obtain slightly different piecewise-linear maps
$\rowH,\proH: \R^{|P|} \rightarrow \R^{|P|}$ that are 
homogeneous versions of $\rowP$, $\proP$.
Then we show (in Section~\ref{sec:birational})
how $\rowH$, $\proH$ can in turn each be viewed as
a tropicalization of a birational map $\rowM$, $\proM$
from a dense open subset $U$ of $\C^{ab}$ to itself;
we call the elements of $U$ {\em $P$-arrays}. 
In Section~\ref{sec:action}, invoking Grinberg and Roby's theorem 
about birational rowmotion and using the recombination method to show 
that birational promotion (like birational rowmotion) is of order $a+b$,
we give a proof of Theorem~\ref{thm:prodhrefined}:
for $v = (v_1,\dots,v_{ab}) \in U$,
the product of the coordinates of $v$ 
associated with elements of the file $S_{\ell}$ (denoted by $|v|_{\ell}$)
has the property that 
$|v|_{\ell} \, |\pro(v)|_{\ell} \,|\pro^2(v)|_{\ell} 
\,\cdots \,|\pro^{n-1}(v)|_{\ell} = 1$.
In Section~\ref{sec:recombine}, we describe recombination
and prove its basic properties.
In Section~\ref{sec:tropic}, we use tropicalization
to deduce from Theorem~\ref{thm:prodhrefined}
a piecewise-linear analogue (Theorem~\ref{thm:sumhrefined})
that by an affine change of variables yields the homomesy result 
for the action of promotion on $\cO(P)$ (Theorem~\ref{thm:sumrefined}).
This last result then yields homomesy for 
the action of promotion on $\cI(P)$ (Theorem~\ref{thm:cardrefined}).
That is, ignoring the use of recombination for passing
back and forth between rowmotion and promotion,
the logic of the argument is that we first prove birational homomesy,
we then deduce piecewise-linear homomesy by tropicalization,
and we finally deduce combinatorial homomesy by specialization.
For completeness, in Section~\ref{sec:opposites}
we use the reciprocity principle of Grinberg-Roby~\cite{GR15}
to prove that the function that sends $f$ to $f(x) f(x')$ 
with $x = (i,j)$ and $x' = (a+1-i,b+1-j)$ is multiplicatively homomesic.

The philosophy of lifting combinatorial actions
to piecewise-linear actions and thence to birational actions
(called ``geometric actions'' by some authors,
as in the phrase ``geometric Robinson-Schensted-Knuth'')
is not original, 
and in particular Kirillov and Berenstein's work 
on operations on Gelfand-Tsetlin patterns~\cite{KB95}
has some parallels with our constructions.
For more background on homomesy,
including several examples different in nature
from the ones considered here but philosophically similar, see~\cite{PR13}.

The authors are grateful to Arkady Berenstein, Darij Grinberg, Michael Joseph, 
Tom Roby, Richard Stanley, and Jessica Striker for helpful conversations
and detailed comments on the manuscript.

\end{section}

\begin{section}{Background}
\label{sec:background}

\begin{subsection}{Homomesy}
\label{subsec:homomesy}

Given a set $X$, an operation $T : X \rightarrow X$ of finite order $n$
(so that $T^n$ is the identity map on $X$),
and a function $F$ from $X$ to a field $\cK$ of characteristic 0, 
we say that $F$ is additively {\em homomesic} 
relative to (or under the action of) $T$, 
or that the triple $(X,T,F)$ exhibits additive {\em homomesy}, 
if, for all $x \in X$, the average of $F$ on the $T$-orbit of $x$
equals some constant $c$ (independent of $x$);
that is, if $(F(x)+F(T(x))+F(T^2(x))+\cdots+F(T^{n-1}(x)))/n = c$.
We also say in this situation that the function $F$ 
(which in this context we will sometimes call a {\em statistic} on $X$)
is $c$-{\em mesic} relative to the map $T$.
We will apply this notion in situations
where $T$ is piecewise-linear (or when $X$ is finite).

In situations where $T$ is birational,
we will use a multiplicative analogue of this notion.
If $F$ is positive throughout $X$,
then we say $F$ is multiplicatively homomesic if
its geometric mean is the same on every orbit.
More generally, $F$ is multiplicatively homomesic if
$(F(x)\,F(T(x))\,F(T^2(x))\,\cdots\,F(T^{n-1}(x)))^{1/n}$
is independent of $x$.
(In some settings it is more natural to relax
the assumption of positivity and merely assert that
$F(x)\,F(T(x))\,F(T^2(x))\,\cdots\,F(T^{n-1}(x))$
is independent of $x$, but we will not take that course here.)
Way $F$ is multiplicatively $c$-mesic if
the geometric mean of $F$ on every orbit is $c$.

We will usually omit the qualifiers ``additive'' and ``multiplicative'',
since the context should make clear which meaning is intended
(additive homomesy in the combinatorial and piecewise-linear realms,
multiplicative homomesy in the birational realm).

\end{subsection}

\begin{subsection}{Posets and toggling}
\label{subsec:posets}

We assume readers are familiar with 
the definition of a finite poset $(P, \leq)$,
as for instance given in Ch.\ 3 of~\cite{Sta11}.
For the most part, we are studying the case $P = [a] \times [b] = 
\{(i,j) \in \N \times \N: \ 1 \leq i \leq a, \ 1 \leq j \leq b\}$
with ordering defined by $(i,j) \leq (i',j')$ 
iff $i \leq i'$ and $j \leq j'$.  We put $n=a+b$. 

We write $x \coveredby y$ (``$x$ is covered by $y$'')
or equivalently $y \covers x$ (``$y$ covers $x$'') 
when $x < y$ and no $z \in P$ satisfies $x < z < y$.
We call $S \subseteq P$ a {\em filter} 
(or {\em upset} or {\em dual order ideal}) of $P$
when $x \in S$ and $y \geq x$ imply $y \in S$.
We call $S \subseteq P$ an {\em order ideal} (or {\em ideal} or {\em downset}) of $P$
when $x \in S$ and $y \leq x$ imply $y \in S$.
The set of filters (resp.\ order ideals) of $P$ 
is denoted by $\cF(P)$ (resp.\ $\cI(P)$).

Following Cameron and Fon-Der-Flaass~\cite{CF95}
and Striker and Williams~\cite{SW12}
we define {\em toggle} operations on $\cF(P)$ and $\cI(P)$.
We treat $\cI(P)$ first.

Given $x \in P$ and $I \in \cI(P)$, let $\tau_x(I)$ 
(``$I$ toggled at $x$'' in Striker and Williams' terminology)
denote the set $I \symmdiff \{x\}$ if this set is in $\cI(P)$ and $I$ otherwise
(where $X \symmdiff Y$ denotes the symmetric difference of the sets $X$ and $Y$).
Equivalently, $\tau_x(I)$ is $I$
unless $y \in I$ for all $y \coveredby x$
and $y \not\in I$ for all $y \covers x$,
in which case $\tau_x(I)$ is $I \symmdiff \{x\}$.
(We will sometimes say ``toggling $x$ turns $I$ into $\tau_x(I)$.'')
Clearly $\tau_x$ is an involution.
It is also easy to show that $\tau_x$ and $\tau_y$ commute
unless $x \coveredby y$ or $x \covers y$.
Cameron and Fon-Der-Flaass proved that
if $x_1,x_2,\dots,x_{|P|}$ is any {\em linear extension} of $P$
(that is, a listing of the elements of $P$ such that
$x_i < x_j$ implies $i<j$), then the composition 
$\tau_{x_1} \circ \tau_{x_2} \circ \cdots \circ \tau_{x_{|P|}}$
(``toggling from top to bottom'') 
is independent of the choice of linear ordering;
we denote it by $\row$.
In the case where the poset $P$ is graded
(that is, where the elements can be partitioned into integer-indexed ranks 
such that $x \coveredby y$ implies that the rank of $x$ is 1 less than the rank of $y$),
one natural way to linearly extend $P$ is to list the elements by rank,
starting with the lowest rank and working upward.
Given the right-to-left order of composition of
$\tau_{x_1} \circ \tau_{x_2} \circ \cdots \circ \tau_{x_{|P|}}$,
this corresponds to toggling the top rank first,
then the next-to-top rank, and so on, lastly toggling the bottom rank.
Note that when $x,y$ belong to the same rank of $P$,
the toggle operations $\tau_x$ and $\tau_y$ commute,
so even without availing ourselves of the theorem of Cameron and Fon-Der-Flaass
we can see that this composite operation on $\cI(P)$ 
(``toggling by ranks from top to bottom'') is well-defined.
Striker and Williams, in their theory of rc posets, 
use the term ``row'' as a synonym for ``rank'',
and they refer to $\row$ as {\em rowmotion}.

For example, let $P = [2] \times [2]$,
and write $(1,1),(2,1),(1,2),(2,2)$ as $w,x,y,z$ for short,
with $w < x < z$ and $w < y < z$ in $P$,
and with the rc embedding shown in Figure~\ref{fig:2x2}.
Under the action of $\tau_z$, $\tau_y$, $\tau_x$, and $\tau_w$,
the order ideal $\{w,x\}$ gets successively mapped to
$\{w,x\}$, $\{w,x,y\}$, $\{w,y\}$, and $\{w,y\}$.
Thus $\row(\{w,x\}) = \{w,y\}$.
\begin{figure}[ht] 
\[
\xymatrixrowsep{0.8pc}\xymatrixcolsep{0.60pc}\xymatrix{
&           & z \sh & \\
& x \nh \sh &       & y \\
&           & w \nh &                        &
}
\]
\caption{The poset $[2] \times [2]$.}
\label{fig:2x2}
\end{figure}

\begin{figure}[ht]
\begin{center}
\begin{pspicture}(-4,-1)(4,7)
\psline[linestyle=dotted](0,0)(4,0)
\psline[linestyle=dotted](-1,1)(4,1)
\psline[linestyle=dotted](-2,2)(4,2)
\psline[linestyle=dotted](-1,3)(4,3)
\psline[linestyle=dotted](0,4)(4,4)
\psline[linestyle=dotted](1,5)(4,5)
\psline[linestyle=dotted](-2,2)(-2,6)
\psline[linestyle=dotted](-1,1)(-1,6)
\psline[linestyle=dotted](0,0)(0,6)
\psline[linestyle=dotted](1,1)(1,6)
\psline[linestyle=dotted](2,2)(2,6)
\psline[linestyle=dotted](3,3)(3,6)
\rput(-3,6.5){{\rm files:}}
\rput(-2,6.5){{1}}
\rput(-1,6.5){{2}}
\rput(0,6.5){{3}}
\rput(1,6.5){{4}}
\rput(2,6.5){{5}}
\rput(3,6.5){{6}}
\rput(4.2,6){{\rm ranks:}}
\rput(4.5,5){{6}}
\rput(4.5,4){{5}}
\rput(4.5,3){{4}}
\rput(4.5,2){{3}}
\rput(4.5,1){{2}}
\rput(4.5,0){{1}}
\psline(-2,2)(-1,1)
\psline(-2,2)(-1,3)
\psline(-1,1)(0,0)
\psline(-1,1)(0,2)
\psline(-1,3)(0,2)
\psline(-1,3)(0,4)
\psline(0,0)(1,1)
\psline(0,2)(1,1)
\psline(0,2)(1,3)
\psline(0,4)(1,3)
\psline(0,4)(1,5)
\psline(1,1)(2,2)
\psline(1,3)(2,2)
\psline(1,3)(2,4)
\psline(1,5)(2,4)
\psline(2,2)(3,3)
\psline(2,4)(3,3)
\pscircle[fillstyle=solid,fillcolor=white](-2,2){.15}
\pscircle[fillstyle=solid,fillcolor=white](-1,1){.15}
\pscircle[fillstyle=solid,fillcolor=white](-1,3){.15}
\pscircle[fillstyle=solid,fillcolor=white](0,0){.15}
\pscircle[fillstyle=solid,fillcolor=white](0,2){.15}
\pscircle[fillstyle=solid,fillcolor=white](0,4){.15}
\pscircle[fillstyle=solid,fillcolor=white](1,1){.15}
\pscircle[fillstyle=solid,fillcolor=white](1,3){.15}
\pscircle[fillstyle=solid,fillcolor=white](1,5){.15}
\pscircle[fillstyle=solid,fillcolor=white](2,2){.15}
\pscircle[fillstyle=solid,fillcolor=white](2,4){.15}
\pscircle[fillstyle=solid,fillcolor=white](3,3){.15}
\rput(-2.02,1.4){{(3,1)}}
\rput(-1.02,0.4){{(2,1)}}
\rput(-1.02,2.4){{(3,2)}}
\rput(-.02,-.6){{(1,1)}}
\rput(-.02,1.4){{(2,2)}}
\rput(-.02,3.4){{(3,3)}}
\rput(0.98,0.4){{(1,2)}}
\rput(0.98,2.4){{(2,3)}}
\rput(0.98,4.4){{(3,4)}}
\rput(1.98,1.4){{(1,3)}}
\rput(1.98,3.4){{(2,4)}}
\rput(2.98,2.4){{(1,4)}}
\end{pspicture}
\end{center}
\caption{The standard rc embedding of the poset $P = [3] \times [4]$,
and the associated partitions of $P$ into ranks and files.  Diagonal
edges are associated with the covering relation in $P$.}
\label{fig:rc}
\end{figure}

We will not go into the general theory of rc posets,
as most of our work concerns 
the special case of rc posets of the form $[a] \times [b]$.
We define the {\em rank} of $(i,j) \in [a] \times [b]$ to be $i+j-1$,
so that in particular the bottom element $(1,1)$ has rank 1
and the top element $(a,b)$ has rank $a+b-1=n-1$
(later we will introduce an extension of $P$
whose bottom element $\hat{0}$ has rank 0
and whose top element $\hat{1}$ has rank $n$).
As an rc poset $[a] \times [b]$ admits an embedding in the plane
that maps $(i,j)$ to the point $(j-i,i+j-2)$;
all poset-elements of rank $m$
are mapped into the horizontal line at height $m-1$ above the origin.
We refer to elements of $P$ that lie on a common vertical line
as belonging to the same {\em file}.
In particular, we say $(i,j)$ belongs to the 
$(j-i+a)$th\footnote{Note that 
as $i$ ranges from $1$ to $a$ and $j$ ranges from $1$ to $b$, 
$j-i+a$ ranges from $1$ to $a+b-1=n-1$; 
this is slightly different from 
the indexing in~\cite{PR13}.} file of $P$.
See Figure~\ref{fig:rc}.
Note that if $x$ and $y$ belong to the same file,
the toggle operations $\tau_x$ and $\tau_y$ commute,
since neither of $x,y$ can cover the other.
Thus the composite operation of toggling the elements of $P$
from left to right is well-defined;
Striker and Williams call this operation {\em promotion},
and we denote it by $\pro$.

We have explained rowmotion and promotion
as toggling by ranks (from top to bottom) 
and by files (from left to right).
A different way to understand these operations
is by toggling {\em fibers}\footnote{We prefer
to avoid words like ``row'', ``column'' and ``diagonal''
since each of these words has two different meanings
according to whether one imbeds $[a] \times [b]$ as an rc poset
or as a subposet of the first quadrant.}
of $[a] \times [b]$,
that is, sets of the form 
$[a] \times \{j\}$ or $\{i\} \times [b]$.
(This view of rowmotion and promotion goes back to Striker and Williams;
see the proof of Theorem 5.4 in ~\cite{SW12}.)
We refer to fibers as being positive or negative
according to whether their slope in the rc embedding
is $+1$ or $-1$.

We illustrate these ideas in the context of $P = [3] \times [3]$.
We consider two orders (both linear extensions of $P$)
in which one can toggle all the elements of $P$ to obtain $\row$.
The first order is just the standard order for rowmotion 
(from top to bottom rank by rank, from left to right within each rank).
The second order toggles the elements in the topmost positive fiber from top to bottom,
then the elements in the middle positive fiber from tom to bottom,
and then the elements in the bottommost positive fiber from top to bottom.
Since the element of $P$ marked $3$ in the left frame
neither covers nor is covered by the element of $P$ marked $4$,
the associated toggles commute, and the same goes for the two toggles
associated with the elements of $P$ marked $6$ and $7$.
Hence the composite operation on the left (``rowmotion by ranks'')
coincides with the composite operation on the right (``rowmotion by fibers'').
This holds for all $a,b$ by the theorem of Cameron and Fon-Der-Flaass~\cite{CF95}.
\begin{figure}[ht]
\begin{center}
\[
\xymatrixrowsep{0.8pc}\xymatrixcolsep{0.80pc}\xymatrix{
   &   & 1 \sh &   &   &       &   &   & 1 \sh &   &   \\
   & 2 \nh \sh &   & 3 \sh &   &       &   & 2 \nh \sh &   & 4 \sh &   \\ 
 4 \nh \sh &   & 5 \nh \sh &   & 6 & \ \ \ & 3 \nh \sh &   & 5 \nh \sh &   & 7 \\
   & 7 \nh \sh &   & 8 \nh &   &       &   & 6 \nh \sh &   & 8 \nh &   \\
   &   & 9 \nh &   &   &       &   &   & 9 \nh &   &  
}
\]
\end{center}
\caption{Rowmotion by ranks and fibers.}
\label{fig:rrf}
\end{figure}

There is a similar picture for promotion.
Here we consider two other orders (not linear extensions of $P$)
in which one can toggle all the elements of $P$.
The first is just the standard order for promotion 
(from left to right file by file, from top to bottom within each file).
The second toggles 
the elements in the topmost positive fiber from left to right, 
then the elements in the middle positive fiber from left to right, 
and then the elements in the bottommost positive fiber from left to right.
As before, the 3 and the 4 can be swapped,
as can the 6 and the 7.
Hence the composite operation on the left (``promotion by files'')
coincides with the composite operation on the right (``promotion by fibers'').
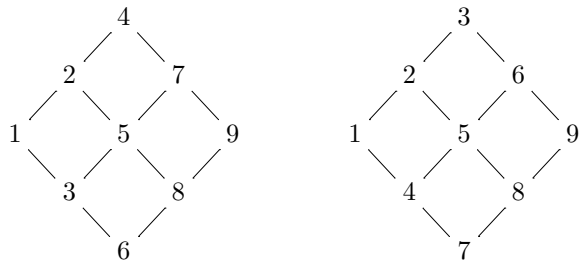
\begin{figure}[ht]
\begin{center}
\[
\xymatrixrowsep{0.8pc}\xymatrixcolsep{0.80pc}\xymatrix{
   &   & 4 \sh &   &   &       &   &   & 3 \sh &   &   \\
   & 2 \nh \sh &   & 7 \sh &   &       &   & 2 \nh \sh &   & 6 \sh &   \\ 
 1 \nh \sh &   & 5 \nh \sh &   & 9 & \ \ \ & 1 \nh \sh &   & 5 \nh \sh &   & 9 \\
   & 3 \nh \sh &   & 8 \nh &   &       &   & 4 \nh \sh &   & 8 \nh &   \\
   &   & 6 \nh &   &   &       &   &   & 7 \nh &   &  
}
\]
\end{center}
\caption{Promotion by ranks and fibers.}
\label{fig:prf}
\end{figure}
Note that for both rowmotion and promotion
we divide the poset into positive fibers and toggle them from top to bottom;
the only difference is whether we toggle the elements within each fiber
from left to right (promotion) or right to left (rowmotion).
The proof that promotion by files and promotion by fibers coincide
for all $a,b$ uses the same commutation results as in the case of rowmotion;
it is just a matter of turning the picture on its side.

So far we have limited ourselves to toggling, rowmotion, and promotion of order ideals;
the definitions for filters are very similar.
Given $x \in P$ and $F \in \cF(P)$, let $\tau_x(F)$ 
be $F \symmdiff \{x\}$ if this set is in $\cF(P)$ and $F$ otherwise.
Rowmotion is defined as toggling from top to bottom
and promotion is defined as toggling from left to right.
When we wish to be ultra-clear about context
we may write order-ideal rowmotion and promotion as $\row_{\cI}$ and $\pro_{\cI}$
and write filter rowmotion and promotion as $\row_{\cF}$ and $\pro_{\cF}$,
but often we will omit the subscripts.

Figure~\ref{fig:filters} shows what $\row_{\cF}$ looks like
as a map on the six-element set $\cF$
(we represent each filter by its indicator function).
\begin{figure}[ht]
\begin{center}
\[
\xymatrixrowsep{0.6pc}\xymatrixcolsep{0.70pc}\xymatrix{
          &          & 1 \sh&     &     &           & 1 \sh &   &     &           & 1 \sh &   &     &           & 0 \sh &   &     &    \\
|\!|\!\!: & 1 \nh\sh &      &  1  & \rFM & 1 \nh \sh &       & 1 & \rFM & 0 \nh \sh &       & 0 & \rFM & 0 \nh \sh &       & 0 & \rFM & :\!\!|\!| \\
          &          & 1 \nh&     &     &           & 0 \nh &   &     &           & 0 \nh &   &     &           & 0 \nh &   &     &
          & & & & & & & & & & & & & & & & & \\
          & & & & & & & & & & & & & & & & & \\
          &          &      &     &     &           & 1 \sh &   &     &           & 1 \sh &   &     &           &       &   &     &    \\
  & & &       & |\!|\!\!:                & 1 \nh \sh &       & 0 & \rFM & 0 \nh \sh &       & 1 & \rFM & :\!\!|\!| &       &   &     &    \\
          &          &      &     &     &           & 0 \nh &   &     &           & 0 \nh &   &     &           &       &   &     &
          }
\]
\end{center}
\caption{The orbits of $\rowF$ on the filters of $[2] \times [2]$.}
\label{fig:filters}
\end{figure}
The matching $|\!|\!\!:$ and $\,:\!\!|\!|$ symbols, borrowed from music notation,
indicate the orbit structure; applying the map $\rho$
to the last listed filter (preceding the $\,:\!\!|\!|$)
brings us back to the first (following the $|\!|\!\!:$).

\end{subsection}

\end{section}

\begin{section}{Piecewise-linear toggling}
\label{sec:PL}


Given a poset $P$ with elements $x_1,\dots,x_{|P|}$, 
let $\R^P$ denote the set of functions $f : P \rightarrow \R$;
we can represent such an $f$
as a {\em $P$-array} (or {\em array} for short)
in which the values of $f(x)$ (for all $x \in P$) are arranged
according to a Hasse diagram for $P$.
We will assume that the elements of $P$
are listed in order of rank (from lowest to highest),
and when $P$ has an rc embedding,
we will assume for convenience that the elements within each rank
are listed from left to right.
We will sometimes identify $\R^P$ with $\R^{|P|}$,
associating $f \in \R^P$ with $v = (f(x_1),\dots,f(x_{|P|}))$.
(We will use ``$f$'' when we want to view $P$-arrays 
as functions that take arguments and return values,
and ``$v$'' when we want to view $P$-arrays
as points in Euclidean space.)
Let $\hat{P}$ denote the augmented poset obtained from $P$ 
by adding two extra elements $\hat{0}$ and $\hat{1}$ 
(which we sometimes denote by $x_0$ and $x_{|P|+1}$)
satisfying $\hat{0} < x < \hat{1}$ for all $x \in P$. 
The {\em order polytope} $\cO(P) \subset \R^P$ 
(see~\cite{Sta86})
is the set of vectors $(\hat{f}(x_1),\dots,\hat{f}(x_{|P|}))$ 
arising from functions $\hat{f} : \hat{P} \rightarrow \R$ 
that satisfy $\hat{f}(\hat{0}) = 0$ and $\hat{f}(\hat{1}) = 1$ 
and are {\em order-preserving}
($x \leq y$ in $\hat{P}$ implies $\hat{f}(x) \leq \hat{f}(y)$ in $\R$). 
Note that the vertices of the order polytope
are the indicator functions of filters.
In some cases it is better to work with the augmented vector 
$(\hat{f}(x_0),\hat{f}(x_1),\dots,\hat{f}(x_{|P|}),\hat{f}(x_{|P|+1}))$.
In either case we have a convex compact polytope.

For example, if $P = [2] \times [2] = \{w,x,y,z\}$ 
as depicted in Subsection~\ref{subsec:posets},
so that $(x_1,x_2,x_3,x_4)=(w,x,y,z)$,
then $\cO(P)$ is the set of vectors $v=(v_1,v_2,v_3,v_4) \in \R^4$ 
such that $0 \leq v_1$, $v_1 \leq v_2 \leq v_4$, 
$v_1 \leq v_3 \leq v_4$, and $v_4 \leq 1$.
It can also be written as the convex hull of the vectors
$(0,0,0,0)$, $(0,0,0,1)$, $(0,0,1,1)$, 
$(0,1,0,1)$, $(0,1,1,1)$, and $(1,1,1,1)$,
which are precisely the vectors associated with the filters of $P$.
It is shown in~\cite{Sta86} that for any poset $P$,
the vertices of $\cO(P)$ correspond to the indicator functions
of the filters of $P$.


We begin our discussion of toggling informally; for formal definitions, see 
Definitions~\ref{def:PLtoggle},~\ref{def:PLrow}, and~\ref{def:PLpro}.
Given a convex compact polytope $\polytope$ in $\R^{|P|}$
(we are only concerned with the case $\polytope=\cO(P)$ here
but the definition makes sense more generally),
we define the {\em piecewise-linear toggle operation} 
$\phi_i$ ($1 \leq i \leq |P|$)
as the unique map from $\polytope$ to itself
whose action on the 1-dimensional cross-sections of $\polytope$ 
in the $i$th coordinate direction
is the affine map that switches the two endpoints of the cross-section.
(We sometimes call these cross-sections fibers, 
e.g.~in our use of the term fiber-toggle below,
though this use of the word ``fiber''
should not be confused with the use found in \ref{subsec:posets}.)
That is, given $v = (v_1,\dots,v_{|P|}) \in \polytope$,
we define 
\begin{equation}
\label{eq:phiv}
\phi_i(v) = (v_1,\dots,v_{i-1},L_i(v)+R_i(v)-v_i,v_{i+1},\dots,v_{|P|}),
\end{equation}
where the real numbers $L_i(v)$ and $R_i(v)$ (usually $L$ and $R$ for short)
are respectively the left and right endpoints of the set 
$\{t \in \R: \ (v_1,\dots,v_{i-1},t,v_{i+1},\dots,v_{|P|})
\in \polytope\}$,
which is a bounded interval because $K$ is 
convex and compact.\footnote{Note that $L$ and $R$ depend on
$v_1,\dots,v_{i-1},v_{i+1},\dots,v_{|P|}$,
though our notation suppresses this dependence.
It is in some ways unfortunate that for an interval $[a,b]$,
the greatest lower bound $a$ (\resp least upper bound $b$)
is called the left endpoint (\resp right endpoint)
rather than the lower endpoint (\resp upper endpoint);
in the Hasse diagram of $P$ as we have drawn it,
$L$ should be thought as being associated with the downward direction,
while $R$ should be thought as being associated with the upward direction.}
Since $L+R-(L+R-v_i)=v_i$, each toggle operation is an involution.

Similar involutions were studied by Kirillov and Berenstein~\cite{KB95}
in the context of Gelfand-Tsetlin triangles.
Indeed, one can view their action as an instance of our framework,
where instead of looking at the rectangle posets $[a] \times [b]$
one looks at the triangle posets
with elements $\{(i,j): \ 1 \leq i \leq j \leq n\}$
and covering-relations $(i,j-1) \coveredby (i,j)$ ($1 \leq i \leq j \leq n$)
and $(i+1,j+1) \coveredby (i,j)$ ($1 \leq i \leq j \leq n-1$);
Kirillov and Berenstein, in their Definition 0.1, 
use the term ``elementary transformations''
for what we call ``fiber-toggles''.

In the case where $\polytope$ is the order polytope of $P$,
and a particular element $x \in P$ has been indexed as $x_i$,
we write $\phi_i$ as $\phi_x$. 
The $L$ and $R$ that appear in (\ref{eq:phiv})
are given by
\begin{equation}
\label{eq:L}
L = \max \{v_j: \ 0 \leq j \leq |P|+1, \ x_j \coveredby x_i\}
\end{equation}
and
\begin{equation}
\label{eq:R}
R = \min \{v_j: \ 0 \leq j \leq |P|+1, \ x_j \covers x_i\} .
\end{equation}
Using these definitions of $L$ and $R$,\footnote{For all
$v$ in $\cO(P)$ one also has 
$L = \max \{v_j : x_j < x_i\}$ 
and 
$R = \min \{v_j : x_j > x_i\}$, 
but the formulas (\ref{eq:L}) and (\ref{eq:R}) 
turn out to be the right ones to use 
in the complement of the order polytope in $\R^P$,
as well as the right ones vis-a-vis lifting
the action to the birational setting.}
we see that equation (\ref{eq:phiv}) defines 
an involution on all of $\R^P$, not just $\cO(P)$.
It is easy to show that $\phi_x$ and $\phi_y$ commute
unless $x \coveredby y$ or $x \covers y$.
These piecewise-linear toggle operations $\phi_x$ are analogous to
the combinatorial toggle operations $\tau_x$ 
(and indeed $\phi_x$ generalizes $\tau_x$ in a sense to be made precise below),
so it is natural to define 
piecewise-linear rowmotion $\rowP: \cO(P) \rightarrow \cO(P)$
as the composite operation accomplished by toggling from top to bottom
(much as filter rowmotion $\row_{\cF}: \cF(P) \rightarrow \cF(P)$
can be defined as the composite operation obtained by
toggling from top to bottom).  
Likewise, if $P$ comes equipped with an rc embedding
(as is the case for $P = [a] \times [b]$), we can define
piecewise-linear promotion $\proP: \cO(P) \rightarrow \cO(P)$
as the composite operation accomplished by toggling from left to right.

Continuing the example $P = [2] \times [2] = \{w,x,y,z\}$
from Subsection~\ref{subsec:posets},
let $v = (.1,.2,.3,.4) \in \cO(P)$,
corresponding to the order-preserving function $f$
that maps $w,x,y,z$ to $.1,.2,.3,.4$, respectively.
Under the action of $\phi_z$, $\phi_y$, $\phi_x$, and $\phi_w$,
the vector $v = (.1,.2,.3,.4)$ gets successively mapped to
\begin{eqnarray*}
\phi_z v & = & (.1,\ .2,\ .3,\ \max(.2,.3)+1-.4) \\
& = & (.1,\ .2,\ .3,\ .9), \\
\phi_y \phi_z v & = & (.1,\ .2,\ .1+.9-.3,\ .9) \\
& = & (.1,\ .2,\ .7,\ .9), \\
\phi_x \phi_y \phi_z v & = & (.1,\ .1+.9-.2,\ .7,\ .9)  \\
& = & (.1,\ .8,\ .7,\ .9), \ \mbox{and} \\
\phi_w \phi_x \phi_y \phi_z v & = & (0+\min(.8,.7)-.1,\ .8,\ .7,\ .9) \\
& = & (.6,\ .8,\ .7,\ .9) \\
& = & \rowP v. 
\end{eqnarray*}

Other examples of piecewise-linear rowmotion
can be seen in Figure~\ref{fig:filters}
(since combinatorial rowmotion of filters
is just the special case of piecewise-linear rowmotion
in which all entries in the $P$-array are equal to 0 or 1,
and the associated filter is the preimage of the value 1).

Grinberg and Roby~\cite{GR15} 
have shown that $\rowP$ on $P = [a] \times [b]$ is of order $n = a+b$,
and by applying recombination (see Section~\ref{sec:recombine})
we will conclude that $\proP$ is of order $n$ as well.

The vertices of $\cO(P)$ are precisely the 0,1-valued functions $f$ on $P$
with the property that $x \leq y$ in $P$ implies $f(x) \leq f(y)$ in $\{0,1\}$.
Each toggle operation acts as a permutation on 
the set of vertices of $\cO(P)$. 
Indeed, if we think of each vertex $\cO(P)$
as determining a cut of the poset $P$
into an upset (filter) $\Sup$ and a downset (order ideal) $\Sdown$
(the pre-image of 1 and 0, respectively,
under the order-preserving map from $P$ to $\{0,1\}$),
then the effect of the toggle operation $\phi_x$ ($x \in P$)
is just to move $x$ from $\Sup$ to $\Sdown$ (if $x$ is in $\Sup$)
or from $\Sdown$ to $\Sup$ (if $x$ is in $\Sdown$)
unless this would violate the property
that $\Sup$ must remain an upset and $\Sdown$ must remain a downset.
In particular, we can see that when our point $v \in \cO(P)$ is a vertex
associated with the cut $(\Sup,\Sdown)$,
the effect of $\phi_x$ on $\Sdown$ is just
toggling the order ideal $\Sdown$ at the element 
$x \in P$.\footnote{This point of view is quite similar to
the monotone Boolean functions point of view
seen in the original literature on what is now called rowmotion.}

It is not hard to show that each toggle operation preserves 
$\min \{f(y) - f(x):$ $x,y \in \hat{P}, \, x \coveredby y\}$.
Therefore $\rowP$ (and $\proP$, when $P$ comes with an rc embedding)
preserve this quantity as well.

In our formal definitions we generalize the above construction
by allowing $f(\hat{0})$ and $f(\hat{1})$
to have fixed values in $\R$ other than 0 and 1 respectively.

\begin{defi} \label{def:PLtoggle}
Suppose $P$ is a poset with $\hat{P}=P \cup \{\hat{0},\hat{1}\}$,
and fix $\sB,\sT$ in $\R$.
We view each $f \in \mathbb{R}^{P}$
as an element of $f \in \mathbb{R}^{\hat{P}}$
by setting $f(\hat{0})=\sB$ and $f(\hat{1})=\sT$.
For $x \in P$,
define $\phi_x f$ as the unique element of $\R^P$ such that
$(\phi_x f)(y) = f(y)$ for all $y \neq x$ in $P$ and
$$(\phi_x f)(x)  
= \max \{f(y): y \in \hat{P}, \ y \coveredby x\}
+ \min \{f(y): y \in \hat{P}, \ y \covers x\} - f(x).$$
\end{defi}

\noindent
(Note that the sets in Definition~\ref{def:PLtoggle}
are guaranteed to be nonempty,
so that the max and min are well-defined.)

\begin{defi} \label{def:PLrow}
With $P$ and $\hat{P}$ as in Definition~\ref{def:PLtoggle},
and for $f$ in $\R^P$,
let $\rowP(f)$ be the element of $\R^P$
obtained by applying to $f$, in succession,
the toggle operations $\phi_{|P|}$,\dots,$\phi_{2}$,$\phi_{1}$, 
where $x_1,x_2,\dots,x_{|P|}$ is some linear extension of $P$
and $\phi_i = \phi_{x_i}$.
An easy adaptation of the proof of Cameron and Fon-Der-Flaass~\cite{CF95}
shows that this operation is independent of the linear ordering.
\end{defi}

For example, returning to the example $P = [2] \times [2] = \{w,x,y,z\}$
from Subsection~\ref{subsec:posets}, let $v = (.1,.2,.3,.4) \in \cO(P)$ as before,
but now set $\alpha=1$ and $\omega=0$ (rather than $\alpha=0$ and $\omega=1$).
Under the action of $\phi_z$, $\phi_y$, $\phi_x$, and $\phi_w$,
the vector $v = (.1,.2,.3,.4)$ gets successively mapped to
\begin{eqnarray*}
\phi_z v & = & (.1,\ .2,\ .3,\ \max(.2,.3)+0-.4) \\
& = & (.1,\ .2,\ .3,\ -.1), \\
\phi_y \phi_z v & = & (.1,\ .2,\ .1+(-.1)-.3,\ -.1) \\
& = & (.1,\ .2,\ -.3,\ -.1), \\
\phi_x \phi_y \phi_z v & = & (.1,\ .1+(-.1)-.2,\ -.3,\ -.1)  \\
& = & (.1,\ -.2,\ -.3,\ -.1), \ \mbox{and} \\
\phi_w \phi_x \phi_y \phi_z v & = & (1+\min(-.2,-.3)-.1,\ -.2,\ -.3,\ -.1) \\
& = & (.6,\ -.2,\ -.3,\ -.1) \\
& = & \rowP v.
\end{eqnarray*}

\begin{defi} \label{def:PLpro}
Take $P$ and $\hat{P}$ as in Definition~\ref{def:PLtoggle},
and suppose we are given an rc embedding of $P$.
For $f$ in $\R^P$,
let $\proP(f)$ be the element of $\R^P$
obtained by applying to $f$, in succession,
the toggle operations from left to right.
\end{defi}

We call $\sB = \sT = 0$ the {\em homogeneous case}
and call $\sB = 0$, $\sT = 1$ the {\em order polytope case}.
In the homogeneous case, we write the rowmotion operation as $\rowH$
(and likewise we write the homogeneous piecewise-linear
promotion operation as $\proH$
when $P$ comes with an rc embedding).
We can always translate by $\alpha$
(replacing $f(x)$ by $f(x)-\alpha$ for all $x \in P$)
to reduce to the case $\alpha = 0$.
In the case where $P$ is graded,
the maps $\rowP$ and $\rowH$ are related
by an affine change of variables,
as are the maps $\proP$ and $\proH$ in the rc case.
Suppose that $\hat{P}$ has $r+1$ ranks, numbered 0 (bottom) through $r$ (top).
Given an arbitrary $f$ in $\R^P$,
define $\widetilde{f}(x) = f(x) - \frac{r-m}{r}\sB - \frac{m}{r}\sT$
for $x$ belonging to rank $m$ ($0 \leq m \leq r$).
Then $\widetilde{f}$ sends $\hat{0}$ and $\hat{1}$ to 0,
and each function from $\hat{P}$ to $\R$
that sends $\hat{0}$ and $\hat{1}$ to 0
arises as $\widetilde{f}$ for a unique $f$ in $\R^P$.
Furthermore, the map $f \mapsto \widetilde{f}$
commutes with rowmotion,\footnote{More generally, 
this way of relating $\rowP$ and $\rowH$
works whenever the poset $P$ is graded.}
so every homomesy for homogeneous rowmotion
gives rise to a (rank-adjusted) homomesy
for order-polytope rowmotion, and vice versa;
the same goes for promotion in the rc case.

Using the homogenization/dehomogenization trick 
one can show that the following two theorems are equivalent:

\begin{theo}
\label{thm:sumrefined}
Fix $a,b \geq 1$, let $n=a+b$, let $P = [a] \times [b]$,
for $1 \leq \ell \leq n-1$
let $|v|_\ell = \sum_x f(x)$
where $x$ ranges over all $(i,j) \in [a] \times [b]$
satisfying $j-i+a=\ell$,
and let $|v|$ be $\sum_{\ell=1}^{n-1} |v|_{\ell}$, 
the sum of all the entries of $v$.
Take $\sB = 0$, $\sT = 1$.
Then for every $v$ in $\R^P$,
and for each $k$ between 1 and $n-1$,
$$\frac{1}{n} \sum_{k=0}^{n-1} |\proP^k (v)|_\ell = 
\left\{ \begin{array}{ll}
a\ell/n & \mbox{if $\ell \leq b$}, \\ 
b(n-\ell)/n & \mbox{if $\ell \geq b$}. \end{array} \right.$$
Summing over $\ell$, we obtain
$$ \frac{1}{n} \sum_{k=0}^{n-1} |\proP^k (v)| = ab/2.$$
\end{theo}
\noindent
(The reason we have different formulas
in Theorems~\ref{thm:cardrefined} and \ref{thm:sumrefined}
is that the former concerns promotion of order ideals
while the latter concerns promotion in the order polytope.
If we replaced order ideals by filters
in the statement of Theorem~\ref{thm:cardrefined},
the formula for the orbit-average would become
what we see in Theorem~\ref{thm:sumrefined}.)

\begin{theo}
\label{thm:sumhrefined}
Fix $a,b \geq 1$, let $n=a+b$, let $P = [a] \times [b]$,
for $1 \leq \ell \leq n-1$
let $|v|_\ell = \sum_x f(x)$
where $x$ ranges over all $(i,j) \in [a] \times [b]$
satisfying $j-i+a=\ell$,
and let $|v|$ be $\sum_{\ell=1}^{n-1} |v|_{\ell}$, 
the sum of all the entries of $v$.
Take $\sB = \sT = 0$.
Then for every $v$ in $\R^P$,
and for each $k$ between 1 and $n-1$,
$$\sum_{k=0}^{n-1} |\proH^k (v)|_\ell = 0.$$
Summing over $\ell$, we obtain
$$\sum_{k=0}^{n-1} |\proH^k (v)| = 0.$$
\end{theo}

We will obtain both Theorem~\ref{thm:sumrefined}
and Theorem~\ref{thm:sumhrefined} 
as consequences of Theorem~\ref{thm:prodhrefined}.
(We do not state a fully general version of 
Theorem~\ref{thm:sumrefined}
with arbitrary values of $\sB$ and $\sT$,
but it is easy to obtain such a result by tropicalizing Theorem~\ref{thm:prodrefined}.)

Note that if $a$ or $b$ is 1,
so that $[a] \times [b]$ is a chain with $n-1$ elements,
then our piecewise-linear maps are all affine,
and if we set $v_0 = v_n = 0$,
then the effect of the map $\phi_i$ ($1 \leq i \leq n-1$) 
is to swap the $i$th and $i+1$st elements of
the difference-vector $(v_1-v_0,v_2-v_1,\dots,v_{n}-v_{n-1})$,
a vector of length $n$ whose entries sum to 0.
Consequently $\proH$ is a cyclic shift of a vector
whose entries sum to 0,
and the claim of Theorem~\ref{thm:sumhrefined} follows in this special case.

It might be possible to prove Theorem~\ref{thm:sumrefined} in the general case
by figuring out how the map $\proP$ dissects the order polytope into pieces
and re-arranges them via affine maps.
Likewise, it might be possible to prove Theorem~\ref{thm:sumhrefined} 
by giving a precise analysis 
of the piecewise-linear structure of the map $\proH$.
However, we will take a different approach,
proving the result in the piecewise-linear setting
by proving it in the birational setting and then tropicalizing.

\end{section}

\begin{section}{Birational toggling}
\label{sec:birational}

The definition of the toggling operation
involves only addition, subtraction, min, and max.
One can define birational transformations on $(\R^+)^P$
that have some formal resemblance to the toggle operations on $\cO(P)$.
This transfer makes use of a dictionary in which 
0, addition, subtraction, max, and min are respectively replaced by 
1, multiplication, division, addition, and parallel addition (defined below),
resulting in a subtraction-free rational expression.\footnote{The 
authors are indebted to Arkady Berenstein for pointing out
the details of this transfer of structure
from the piecewise-linear setting to the birational setting. 
A key signpost pointing in the correct direction 
was his remark that $\min(x,y) = -\max(-x,-y)$.}
Parallel addition can be expressed in terms of the other operations,
but taking a symmetrical view of the two forms of addition 
turns out to be fruitful.
Indeed, in setting up the correspondence we have a choice to make:
by series-parallel duality,
one could equally well use a dictionary 
that switches the roles of addition and parallel addition.

For $x,y$ satisfying $x+y \neq 0$, we define 
the parallel sum of $x$ and $y$ as $x \psum y = xy/(x+y)$.
In the case where $x$, $y$ and $x+y$ are all nonzero,
$xy/(x+y)$ is equal to $1/(\frac1x+\frac1y)$,
which clarifies the choice of notation and terminology:
if two electrical resistors of resistance $x$ and $y$ are connected in parallel,
the resulting circuit element has effective resistance $x \psum y$.
Note that if $x$ and $y$ are in $\R^+$, 
then $x+y$ and $x \psum y$ are in $\R^+$ as well.
Also note that $\!\psum\!$ is commutative and associative,
so that a compound parallel sum $x \psum y \psum z \psum \cdots$
is well-defined; it is equal to product $x y z \cdots$
divided by the sum of all products that omit exactly one of the variables,
and in the case where $x,y,z,\dots$ are all nonzero,
it can also be written as $1/(\frac1x+\frac1y+\frac1z+\cdots)$.
We thus have the reciprocity relation
\begin{equation}
\label{eq:reciprocity}
\left( x \psum y \psum z \psum \cdots \right) 
\left( \frac{1}{x}+\frac{1}{y}+\frac{1}{z}+\cdots \right) = 1.
\end{equation}
The identity 
\begin{equation}
\label{eq:duality}
(x \psum y) (x+y) = xy
\end{equation}
will also play an important role.

As in the previous section, we begin informally,
preparing for Definitions~\ref{def:btoggle},~\ref{def:brow}, and~\ref{def:bpro}.
Recall formulas (\ref{eq:phiv}), (\ref{eq:L}) and (\ref{eq:R}) above.
Instead of taking the maximum of the $v_j$'s satisfying $x_j \coveredby x_i$,
we take their ordinary (or ``series'') sum,
and instead of taking the minimum of the $v_j$'s 
satisfying $x_j \covers x_i$,
we take their parallel sum.
Proceeding formally, given a nonempty multiset $S = \{s_1,s_2,\dots\}$,
let $\bigsum S$ denote $s_1 + s_2 + \cdots$ 
and $\bigpsum S$ denote $s_1 \psum s_2 \psum \cdots$.
(To see why we must think of $S$ as a multiset
and take multiplicity into account,
consider for instance the series sum and parallel sum
in equations (\ref{eq:Lb}) and (\ref{eq:Rb}) respectively;
if we happen to have $v_{j_1} = v_{j_2}$
for some pair $j_1 \neq j_2$, both of the terms must be included.)
Then for $v = (v_0,v_1,\dots,v_{|P|},v_{|P|+1}) \in (\R^+)^{\hat{P}}$ 
with $v_0 = \alpha = 1$ and $v_{|P|+1} = \omega = 1$
(recall that $x_0 = \hat{0}$ and $x_{|P|+1} = \hat{1}$)
and for $1 \leq i \leq |P|$ we take
\begin{equation}
\label{eq:phib}
\phi_i(v) = (v_0,v_1,\dots,v_{i-1},LR/v_i,v_{i+1},\dots,v_{|P|},v_{|P|+1}),
\end{equation}
with
\begin{equation}
\label{eq:Lb}
L = \bigsum \{v_j: \ 0 \leq j \leq |P|+1, \ x_j \coveredby x_i\}
\end{equation}
and
\begin{equation}
\label{eq:Rb}
R = \bigpsum \{v_j: \ 0 \leq j \leq |P|+1, \ x_j \covers x_i\}
\end{equation}
where as in Section~\ref{sec:PL}
the sets in question are guaranteed to be nonempty.
We call the maps $\phi_i$ given by (\ref{eq:phib}) 
{\em birational toggle operations},
as opposed to the piecewise-linear toggle operations
treated in the previous section.\footnote{We use 
the same symbol $\phi_i$ for both,
but this should cause no confusion.}
As the $0$th and $|P|+1$st coordinates of $v$ 
are not affected by any of the toggle operations
we can just omit those coordinates,
reducing our toggle operations to actions on $(\R^+)^P$.
Since $LR/(LR/v_i) = v_i$,
each birational toggle operation is an involution
on the orthant $(\R^+)^{P}$.
The birational toggle operations are analogous to
the piecewise-linear toggle operations
(in a sense to be made precise below),
so it is natural to define 
{\em birational rowmotion} $\rowB: (\R^+)^P \rightarrow (\R^+)^P$
as the composite operation accomplished by toggling from top to bottom,
and to define
{\em birational promotion} $\proB: (\R^+)^P \rightarrow (\R^+)^P$
as the composite operation accomplished by toggling from left to right in the rc case.

It is not hard to show that 
each birational toggle operation preserves 
$\bigsum \{f(x) / f(y): \: x,y \in \hat{P}, \, x \coveredby y\}$
(or, if one prefers, the reciprocal quantity
$\bigpsum \{f(y) / f(x): \: x,y \in \hat{P}, \, x \coveredby y\}$).
Therefore $\rowB$ also preserves this quantity,
as does $\proB$ in the rc case.
(We are indebted to Arkady Berenstein for this observation.)
One could define the {\em birational toggle group}
as the group generated by all birational toggles,
and the {\em piecewise-linear toggle group} analogously.
These are related to the (combinatorial) ``toggle group''
earlier authors have studied (the group induced by
the action of all the toggles on $\cI(P)$),
but unlike the combinatorial toggle group,
these groups are in general infinite.
It seems possible that, at least in some cases,
the birational toggle group contains
all birational transformations that preserve
$\bigsum \{f(x) / f(y): \: x,y \in \hat{P}, \, x \coveredby y\}$.

Continuing our running example $P = [2] \times [2] = \{w,x,y,z\}$,
let $v = (1,2,3,4) \in (\R^+)^P$,
corresponding to the positive function $f$
that maps $w,x,y,z$ to $1,2,3,4$, respectively.
Under the action of $\phi_z$, $\phi_y$, $\phi_x$, and $\phi_w$,
the vector $v = (1,\:2,\:3,\:4)$ gets successively mapped to
\begin{eqnarray*}
\phi_z v & = & (1,\ 2,\ 3,\ (2+3)(1)/4) \\
& = & (1,\ 2,\ 3,\ {\textstyle \frac{5}{4}}), \\
\phi_y \phi_z v & = & (1,\ 2,\ (1)({\textstyle \frac{5}{4}})/3,\ {\textstyle \frac{5}{4}}) \\
& = & (1,\ 2,\ {\textstyle \frac{5}{12}},\ {\textstyle \frac{5}{4}}), \\
\phi_x \phi_y \phi_z v & = & (1,\ (1)({\textstyle \frac{5}{4}})/2,\ {\textstyle \frac{5}{12}},\ {\textstyle \frac{5}{4}})  \\
& = & (1,\ {\textstyle \frac{5}{8}},\ {\textstyle \frac{5}{12}},\ {\textstyle \frac{5}{4}}), \ \mbox{and} \\
\phi_w \phi_x \phi_y \phi_z & = & ((1)({\textstyle \frac{5}{8}} \psum {\textstyle \frac{5}{12}})/1,\ {\textstyle \frac{5}{8}},\ {\textstyle \frac{5}{12}},\ {\textstyle \frac{5}{4}}) \\
& = & ({\textstyle \frac{1}{4}},\ {\textstyle \frac{5}{8}},\ {\textstyle \frac{5}{12}},\ {\textstyle \frac{5}{4}}) \\
& = & \rowB(v).
\end{eqnarray*}
We can check that the quantity 
$\bigsum \{f(x) / f(y): \: x,y \in \hat{P}, x \coveredby y\}$
retains the value $\frac{85}{12}$ throughout the process. 

For most of the purposes of this article,  
it suffices to take $\proB$ to be a map from $(\R^+)^P$ to itself;
since the variables take on only positive values
and since the expressions in those variables are all subtraction-free,
all of the denominators are non-zero,
so none of the rational functions blow up.
Alternatively, as in the work of Grinberg and Roby,
one can replace $\R^+$ by
a ring of rational functions in formal indeterminates
indexed by the elements of $P$,
thereby avoiding the singularity issue.
A third approach is to extend $\rowB$ and $\proB$ to maps
from a dense open subset $U$ of $\C^P$ to itself
by avoiding the points where denominators vanish.
This is slightly subtle, since one needs to exclude
all points whose orbits intersect the singular set.
That is, we need to avoid not merely those points
where the map blows up, but also points 
where the $k$th iterate of the map blows up.
When the map is of finite order,
this means avoiding a finite union of Zariski-closed proper subsets.
It should be possible to characterize 
the resulting dense open set $U$,
but we do not do this here;
we merely point out that such a dense open set must exist
since there are only finitely many denominators.

We now give formal definitions.
Fix $\alpha$, $\omega$ in $\R^+$.

\begin{defi} \label{def:btoggle}
Suppose $P$ is a poset with $\hat{P}=P \cup \{\hat{0},\hat{1}\}$.
For $f$ in $(\R^+)^P$
(extended to an element of $(\R^+)^{\hat{P}}$
by putting $f(\hat{0})=\alpha$ and $f(\hat{1})=\omega$)
and for $x \in P$,
define $\phi_x f$ as the unique element of $(\R^+)^P$ such that
$(\phi_x f)(y) = f(y)$ for all $y \neq x$ in $P$ and
$$(\phi_x f)(x)
= \left( \bigsum \{f(y): y\!\in\!\hat{P}, \ y \coveredby x\} \right)
\!\left( \bigpsum \{f(y): y\!\in\!\hat{P}, \ y \covers x\} \right) / f(x).$$
\end{defi}

\noindent
(Note that the sets in Definition~\ref{def:btoggle}
are guaranteed to be nonempty,
so that the sum and parallel sum are well-defined.)

\begin{defi} \label{def:brow}
With $P$ and $\hat{P}$ as in Definition~\ref{def:btoggle},
and for $f$ in $(\R^+)^P$,
let $\rowB(f)$ be the element of $(\R^+)^P$
obtained by applying to $f$, in succession,
the toggle operations $\phi_{|P|}$,\dots,$\phi_{2}$,$\phi_{1}$,
where $x_1,x_2,\dots,x_{|P|}$ is some linear extension of $P$
and $\phi_i = \phi_{x_i}$.
An easy adaptation of the proof of Cameron and Fon-Der-Flaass~\cite{CF95}
shows that this operation is independent of the linear ordering.
\end{defi}

\begin{defi} \label{def:bpro}
Take $P$ and $\hat{P}$ as in Definition~\ref{def:btoggle},
and suppose we are given an rc embedding of $P$.
For $f$ in $(\R^+)^P$,
let $\proB(f)$ be the element of $(\R^+)^P$
obtained by applying to $f$, in succession,
the toggle operations from left to right.
\end{defi}

We focus mostly on the {\em monic} case
$f(\hat{0}) = f(\hat{1}) = 1$, since
(just as in the piecewise-linear setting) no generality is lost,
if $P$ is graded
and we restrict to vectors $v$ in the positive orthant
so that $r$th roots are globally well-defined,
where $r+1$ is the number of ranks of $\hat{P}$.
Given an arbitrary $f: \hat{P} \rightarrow \R^+$,
let $\alpha = f(\hat{0})$
and $\omega = f(\hat{1})$,
and define $\widetilde{f}$ by 
$\widetilde{f}(x) = f(x) / (\alpha^{1-m/r} \omega^{m/r})$
for $x$ belonging to rank $m$ ($0 \leq m \leq r$).
Then $\widetilde{f}$ sends $\hat{0}$ and $\hat{1}$ to 1,
and for any choice of $\alpha,\omega \in \R^+$,
each function from $\hat{P}$ to $\R^+$
that sends $\hat{0}$ and $\hat{1}$ to 1
arises as $\widetilde{f}$ for a unique $f : \hat{P} \rightarrow \R^+$
sending $\hat{0}$ to $\alpha$ and $\hat{1}$ to $\omega$.
Furthermore, the map $f \mapsto \widetilde{f}$ commutes with rowmotion 
(and with promotion in the rc case),
so homomesy results for $\widetilde{f}$
yield homomesy results for $f$ as immediate consequences.
When we wish to emphasize that we are in
the monic case $\alpha = \omega = 1$,
we will write $\rowB$ and $\proB$ as $\rowM$ and $\proM$.

As in the piecewise-linear case,
our proof-method will enable us to demonstrate
(multiplicative) homomesy of the action
when $P$ is a product of two chains.

\begin{theo}
\label{thm:prodhrefined}
Fix $a,b \geq 1$, let $n=a+b$, let $P = [a] \times [b]$,
for $1 \leq \ell \leq n-1$
let $|v|_\ell = \prod_x f(x)$
where $x$ ranges over all $(i,j) \in [a] \times [b]$
satisfying $j-i+a=\ell$,
and let $|v|$ be $\prod_{\ell=1}^{n-1} |v|_{\ell}$, 
the product of all the entries of $v$.
Then for every $v$ in $(\R^+)^P$
corresponding to an $f: \hat{P} \rightarrow \R^+$
with $f(\hat{0})=f(\hat{1})=1$,
$$\prod_{k=0}^{n-1} |\proM^k (v)|_\ell = 1.$$
Multiplying over $\ell$, we obtain
$$\prod_{k=0}^{n-1} |\proM^k (v)| = 1.$$
\end{theo}

In other words, the geometric mean of the values of $|\proM(v)|_{\ell}$
as $v$ traces out an orbit in $(\R^+)^P$ is equal to 1 for every orbit.
Theorem~\ref{thm:prodhrefined} also applies to a dense open subset of $\R^P$,
and indeed to a dense open subset of $\C^P$,
but the paraphrase in terms of geometric means does not hold in general
since $z \mapsto z^{1/n}$ is not single-valued on $\C$.

We will derive Theorem~\ref{thm:prodhrefined} from Theorem~\ref{thm:prodrefined}.

\end{section}

\begin{section}{File-toggling and promotion}
\label{sec:action}

In this section we assume $P = [a] \times [b]$.
We work in the birational realm,
though we switch to considering promotion rather than rowmotion
for reasons that will soon be clear.
Earlier we defined $S_{\ell}$, 
the $\ell$th file in $[a] \times [b]$ (with $1 \leq \ell \leq n-1$),
as the set of all $(i,j) \in [a] \times [b]$ with $j-i+a=\ell$.

We start in the monic case ($\alpha=\omega=1$),
though later in the section we will consider
the general not-necessarily-monic case.
Given $f: \hat{P} \rightarrow \R^+$
with $f(\hat{0}) = f(\hat{1}) = 1$,
let $p_\ell=|v|_\ell$ ($1 \leq \ell \leq n-1$) 
be the product of the numbers $f(x)$
with $x$ belonging to the $\ell$th file of $P$,
let $p_0 = p_n = 1$,
and for $1 \leq \ell \leq n$ let $q_\ell = p_\ell/p_{\ell-1}$.
Call $q_1,\dots,q_n$ the {\em quotient sequence} associated with $f$,
and denote it by $Q(f)$.
This is analogous to the difference sequence introduced in~\cite{PR13}.
Note that the product $q_1 \cdots q_n$ telescopes to $p_n/p_0=1$.
For $\ell$ between 1 and $n-1$,
let $\phi_\ell^*$ be the product of the commuting involutions $\phi_x$
for all $x$ belonging to the $\ell$th file.
Lastly, given a sequence of $n$ numbers $w = (w_1,\dots,w_n)$,
and given $1 \leq \ell \leq n-1$, define 
$\sigma_\ell (w) = (w_1,\dots,w_{\ell-1},w_{\ell+1},w_{\ell},w_{\ell+2},\dots,w_n)$;
that is, $\sigma_\ell$ switches the $\ell$th and $\ell+1$st entries of $w$.

\begin{lemma}
\label{lem:swap}
For all $1 \leq \ell \leq n-1$,
$$Q(\phi_\ell^* f) = \sigma_\ell Q(f).$$
That is, toggling the $\ell$th file of $f$
swaps the $\ell$th and $\ell+1$st entries
of the quotient sequence of $f$.
\end{lemma}

\begin{proof}
Let $f'=\phi_\ell^* f$,
let $p'_\ell$ ($1 \leq \ell \leq n-1$) be the product of the numbers $f'(x)$
with $x$ belonging to the $\ell$th file of $P$,
let $p'_0 = p'_n = 1$,
and for $1 \leq \ell \leq n$ let $q'_\ell = p'_\ell/p'_{\ell-1}$.
We have $p'_{\ell'} = p_{\ell'}$ for all ${\ell'} \neq \ell$
(since only the values associated with elements of $P$
of the $\ell$th file are affected by $\phi_\ell$),
so we have $q'_{\ell'} = p'_\ell/p'_{\ell-1} = p_\ell/p_{\ell-1} = q_{\ell'}$ 
for all ${\ell'}$ other than $\ell$ and $\ell+1$.
The product $q'_1 \cdots q'_n$ telescopes to 1 as before.
The lemma asserts that $q'_\ell = q_{\ell+1}$ and $q'_{\ell+1} = q_\ell$.
It suffices to prove just one of the two assertions,
since the two previous sentences tell us
that $q'_\ell q'_{\ell+1} = q_\ell q_{\ell+1}$.
Expressed in terms of the $p_{\ell'}$'s, the assertion $q'_\ell = q_{\ell+1}$ 
amounts to the claim $p'_{\ell}/p'_{\ell-1} = p_{\ell+1}/p_{\ell}$,
or equivalently $p_{\ell} p'_{\ell} = p'_{\ell-1} p_{\ell+1}$,
which is just the claim $p_{\ell} p'_{\ell} = p_{\ell-1} p_{\ell+1}$.

We write $p_{\ell} p'_{\ell}$ as the product of
$f(x) f'(x)$ as $x$ varies over the $\ell$th file of $P$.
For each $x$ in the $\ell$th file of $P$, $f(x) f'(x) = L_x R_x$
where $L_x = \bigsum \{f(w): \ w \coveredby x\}$
and $R_x = \bigpsum \{f(y): \ y \covers x\}$ (with $w,y \in \hat{P}$).
Now we note a key property of the structure of $P = [a] \times [b]$:
if $x_+$ and $x_-$ are vertically adjacent elements of the $\ell$th file, 
with $x_+$ above $x_-$,
the $w$'s that contribute to $L_{x_+}$ are precisely
the $y$'s that contribute to $R_{x_-}$.
So, when we take the product of $f(x) f'(x) = L_x R_x$
over all $x$ in the $\ell$th file, the factors $L_{x_+}$ and $R_{x_-}$ combine
to give $\prod_w f(w)$
where $w$ varies over the (two) elements 
satisfying $x_- \coveredby w \coveredby x_+$
(here we are using the identity (\ref{eq:duality})).
The only factors that do not combine in this way are
$L_x$ where $x$ is the bottom element of the $\ell$th file
and $R_x$ where $x$ is the top element of the $\ell$th file.
Both of these factors can be written in the form $f(z)$
where $z$ is a single element of $\hat{P}$
belonging to either the $\ell-1$st or $\ell+1$st file.
By examining cases, it is easy to check
that every element of the $\ell-1$st file or $\ell+1$st file
makes a single multiplicative contribution,
so that $p_{\ell} p'_{\ell}$ is the product of $f(z)$
as $z$ varies over the union of the $\ell-1$st and $\ell+1$st files of $P$.
But this product is precisely $p_{\ell-1} p_{\ell+1}$.
So we have proved that $p_{\ell} p'_{\ell} = p_{\ell-1} p_{\ell+1}$, 
which concludes the proof.
\end{proof}

\begin{coro}
\label{cor:shift}
$Q(\proM f)$ is the leftward cyclic shift of $Q(f)$.
\end{coro}

\begin{proof}
Recall that $\proM$ is the composition
$\phi_{n-1}^* \circ \cdots \circ \phi_{1}^*$.
\end{proof}

\begin{proof}[Proof of Theorem~\ref{thm:prodhrefined}]
For each $\ell$ between 1 and $n$, view $q_\ell$ as a function of $f$.
Corollary~\ref{cor:shift} tells us that the numbers
$q_\ell(\proM^0 f), q_\ell(\proM^1 f), q_\ell(\proM^2 f), \dots, q_\ell(\proM^{n-1} f)$
are respectively equal to the numbers
$q_{\ell}(f),q_{\ell+1}(f), q_{\ell+2}(f), \dots, q_{\ell-1}(f)$, which multiply to 1.
Therefore $q_\ell$ (viewed as a function of $f$)
is multiplicatively homomesic under the action of $\proM$
(with average value 1 on all orbits), for all $1 \leq \ell \leq n$.
Hence the same is true of $p_1 = q_1$,
$p_2 = q_1 q_2$, $p_3 = q_1 q_2 q_3$, etc.,
so that for all $\ell \leq n-1$,
$p_\ell = |v|_\ell$ is multiplicatively 1-homomesic, as claimed.
\end{proof}

In fact, switching now to the general (not necessarily monic) case, we have:

\begin{theo}
\label{thm:prodrefined}
Fix $a,b \geq 1$, let $n=a+b$, let $P = [a] \times [b]$,
and for $1 \leq \ell \leq n-1$
let $|v|_\ell = \prod_x f(x)$
where $x$ ranges over all $(i,j) \in [a] \times [b]$
satisfying $j-i+a=\ell$.
Also take $\alpha,\omega \in \R^+$.
Then for every $v$ in $(\R^+)^P$
corresponding to an $f: \hat{P} \rightarrow \R^+$
with $f(\hat{0})=\alpha$ and $f(\hat{1})=\omega$,
$$\prod_{k=0}^{n-1} |\proB^k (v)|_\ell = 
\left\{ \begin{array}{ll}
\alpha^{b \ell} \omega^{a \ell} & \mbox{if $\ell \leq \min(a,b)$}, \\
\alpha^{a(n - \ell)} \omega^{a \ell} & \mbox{if $a \leq \ell \leq b$}, \\
\alpha^{b \ell} \omega^{b(n - \ell)} & \mbox{if $b \leq \ell \leq a$}, \\
\alpha^{a(n - \ell)} \omega^{b(n - \ell)} & \mbox{if $\ell \geq \max(a,b)$}. 
\end{array} \right.$$
Multiplying over $\ell$, we obtain
$$\prod_{k=0}^{n-1} |\proB^k (v)| = 
\alpha^{nab/2} \omega^{nab/2}.$$
\end{theo}

\begin{proof}[Proof of Theorem~\ref{thm:prodrefined}]
Suppose $a \geq b$ (the proof for the case $a \leq b$ is similar).
Define $\widetilde{f}$ by 
$\widetilde{f}(x) = f(x) / \alpha^{(n-m)/n} \omega^{m/n}$
for $x$ belonging to rank $m$.
Since the statistic $\widetilde{f} \mapsto |\widetilde{f}|_{\ell}$
is multiplicatively 1-mesic under the action of $\proM$,
and since the map $f \mapsto \widetilde{f}$
intertwines $\proM$ and $\proP$,
the statistic $f \mapsto |f|_{\ell}$
is multiplicatively $c_{\ell}$-mesic under the action of $\proM$,
where $c_{\ell}$ is the $n$th root
of the product of $\alpha^{n+1-i-j} \omega^{i+j-1}$
as $(i,j)$ ranges over $S_{\ell}$.
Hence the product in the formula above
is equal to $c_{\ell}^n$.
It is slightly tedious but not hard to check that
for $1 \leq \ell \leq b \leq a$ we have
$c_{\ell}^n = \alpha^{b \ell} \omega^{a \ell}$,
for $b \leq \ell \leq a$ we have
$c_{\ell}^n = \alpha^{b \ell} \omega^{b (n-\ell)}$,
and for $b \leq a \leq \ell \leq n-1$ we have
$c_{\ell}^n = \alpha^{a (n-\ell)} \omega^{b (n-\ell)}$.
We omit the proof of the aggregate formula
(obtained by multiplying over $\ell$).
\end{proof}

Note that Theorem~\ref{thm:prodrefined} 
immediately implies Theorem~\ref{thm:prodhrefined}
via the substitution $\alpha = \omega = 1$. 

Also note that, although the theorem asserts homomesy
only for $v$'s in $(\R^+)^P$,
the same holds for every $v$ in a dense open subset of $\R^P$,
whose complement consists of points
for which the orbit $v,\proB(v),\dots,\proB^{n-1}(v)$
is not well-defined because of some denominator vanishing.
The same is true for $\C^P$.

\end{section}

\begin{section}{Recombination}

\label{sec:recombine}


The ideas of this section apply in all three realms
(combinatorial, piecewise-linear, and birational)
and apply to many posets,
though we will restrict ourselves to the the case $P = [a] \times [b]$.
For the sake of generality, we treat the general birational case,
where $P$-arrays $f$ tacitly have
$f(\hat{0})=\alpha$ and $f(\hat{1})=\omega$.

The essential idea behind recombination is presented in Figure~\ref{fig:r2x2}).
This Figure shows (at top) the four elements 
of a particular $\rho_{\cP}$ orbit for $P = [2] \times [2]$
and (at bottom) the four elements 
of a particular $\pi_{\cP}$ orbit for $P = [2] \times [2]$,
with each element of the first orbit
being mapped via $\mathfrak{R}$ (the recombination operation)
to a corresponding element of the second orbit:
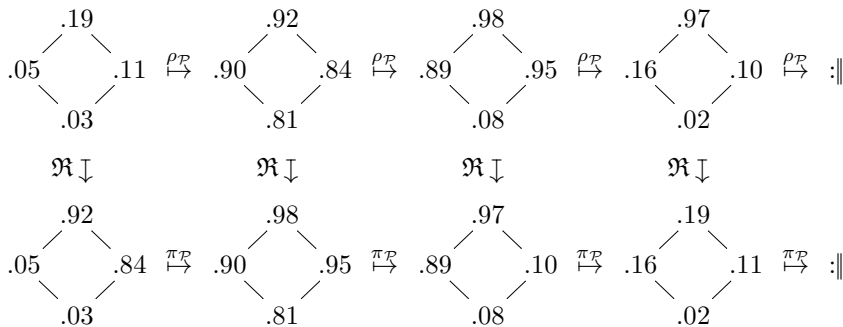
\begin{figure}[ht] 
\[
\xymatrixrowsep{0.3pc}\xymatrixcolsep{0.10pc}\xymatrix{
&          &.19\sh&     &     &           &.92\sh &   &     &           &.98\sh &   &     &           &.97\sh &   &     &    \\
|\!|\!\!: &.05\nh\sh &      & .11 & \nR &.90\nh \sh &       &.84& \nR &.89\nh \sh &       &.95& \nR &.16\nh \sh &       &.10& \nR & :\!\!|\!| \\
&          &.03\nh&     &     &           &.81\nh &   &     &           &.08\nh &   &     &           &.02\nh &   &     &
& & & & & & & & & & & & & & & & & \\
& & \jim & & & & \jim & & & & \jim & & & & \jim & & & \\
&          &.92\sh&     &     &           &.98\sh &   &     &           &.97\sh &   &     &           &.19\sh &   &     &    \\
|\!|\!\!: &.05\nh\sh &      &.84  & \nP &.90\nh \sh &       &.95& \nP &.89\nh \sh &       &.10& \nP &.16\nh \sh &       &.11& \nP & :\!\!|\!| \\
&          &.03\nh&     &     &           &.81\nh &   &     &           &.08\nh &   &     &           &.02\nh &   &     &
}
\]
\caption{Recombination in $[2] \times [2]$.}
\label{fig:r2x2}
\end{figure}
The sixteen numbers that appear in the $\rowP$ orbit
are the same as
the sixteen numbers that appear in the $\proP$ orbit,
in a different order.
Specifically, if $f$ denotes a $P$-array,
then the $P$-array $\recom f$ consists of
the bottom and left entries of $f$ along with
the right and top entries of $\rowP f$.

Recombination is implicit in the work of Striker and Williams~\cite{SW12};
they show (in the combinatorial realm)
that $\rho$ and $\pro$ are conjugate to one another.
However, the study of homomesies
requires that the Striker-Williams conjugation map
be expressed in a more explicit form, as is done here.
The discussion of recombination here is focused on the case $P = [a] \times [b]$,
but the notion applies more generally;
a version of it is discussed in~\cite{E+16} (see Section 6),
and Vorland's article~\cite{V18}
treats a multidimensional version suitable for rowmotion and promotion
in the case where $P$ is a product of more than two chains.

Note that it is not immediately obvious that recombination as defined above
carries the order polytope to itself,
or that it sends filters to filters and order ideals into order ideals.

\begin{figure}[ht] 
\[
\xymatrixrowsep{0.4pc}\xymatrixcolsep{0.40pc}\xymatrix{
&   &   &a_7\sh&   &   &   &   &   &b_7\sh&   &   \\
&   &a_4\nh\sh&   &a_8\sh&   &   &   &b_4\nh\sh&   &b_8\sh&   \\
f=  &a_1\nh\sh&   &a_5\nh\sh&   &a_9&\mR&b_1\nh\sh&   &b_5\nh\sh&   & b_9 \\
&   &a_2\nh\sh&   &a_6\nh&   &   &   &b_2\nh\sh&   &b_6\nh&   \\
&   &   &a_3\nh&   &   &   &   &   &b_3\nh&   &   \\
} 
\]
\[
\xymatrixrowsep{0.4pc}\xymatrixcolsep{0.40pc}\xymatrix{
&   &   &c_7\sh&   &   &   &   &   &d_7\sh&   &  \\
&   &c_4\nh\sh&   &c_8\sh&   &   &   &d_4\nh\sh&   &d_8\sh& \\
\ \mR \ &\ c_1\nh\sh&   &c_5\nh\sh&   &c_9&\mR&d_1\nh\sh&   &d_5\nh\sh&   &d_9 \\
&   &c_2\nh\sh&   &c_6\nh&   &   &   &d_2\nh\sh&   &d_6\nh&   \\
&   &   &c_3\nh&   &   &   &   &   &d_3\nh&   &   \\
}
\]
\caption{A partial orbit of $\rowB$ for $P=[3] \times [3]$.}
\label{fig:orbit}
\end{figure}
\begin{figure}[ht] 
\[
\xymatrixrowsep{0.4pc}\xymatrixcolsep{0.40pc}\xymatrix{
  &  &   &c_7\sh&   &   &   &   &   &d_7\sh&   &   \\
  &  &b_4\nh\sh&   &c_8\sh&   &   &   &c_4\nh\sh&   &d_8\sh&   \\
g=&a_1\nh\sh&  &b_5\nh\sh&   &c_9&\mP&b_1\nh\sh&   &c_5\nh\sh&   &d_9\\
  &  &a_2\nh\sh&   &b_6\nh&   &   &   &b_2\nh\sh&   &c_6\nh&   \\
  &  &   &a_3\nh&   &   &   &   &   &b_3\nh&   &
}
\]
\caption{Recombination.}
\label{fig:fibers}
\end{figure}
\begin{figure}[ht]
\[
\xymatrixrowsep{0.4pc}\xymatrixcolsep{0.40pc}\xymatrix{
&  &   &c_7\sh&   &   &   &   &   &d_7\sh&   &   \\
&  &c_4\nh\sh&   &c_8\sh&   &   &   &c_4\nh\sh&   &c_8\sh&   \\
&c_1\nh\sh&  &b_5\nh\sh&   &b_9&   &b_1\nh\sh&   &b_5\nh\sh&   &c_9\\
&  &b_2\nh\sh&   &b_6\nh&   &   &   &b_2\nh\sh&   &b_6\nh&   \\
&  &   &b_3\nh&   &   &   &   &   &a_3\nh&   &
}
\]
\caption{Key example: 
Toggling the middle entry of
$\tau_{(3,1)} \tau_{(2,3)} \tau_{(3,2)} \tau_{(3,3)} \rho_{\mathcal{B}}(f)$ (left) 
versus toggling the middle entry of
$\tau_{(3,3)} \tau_{(2,1)} \tau_{(3,2)} \tau_{(3,1)} g$ (right).}
\label{fig:key}
\end{figure}

The link between rowmotion and promotion
is seen mostly clearly if we look at the negative fibers.
It is helpful to consider the specific example $P=[3] \times [3]$,
where we implement rowmotion by toggling the $P$-array
row by row from top to bottom and within each row from left to right
(at $(3,3)$, $(3,2)$, $(2,3)$, $(3,1)$, $(2,2)$, $(1,3)$, 
$(2,1)$, $(1,2)$, and $(1,1)$, respectively),
and we implement promotion by toggling the $P$-array
file by file from left to right and within each file from top to bottom
(at $(3,1)$, $(3,2)$, $(2,1)$, $(3,3)$, $(2,2)$, $(1,1)$, 
$(2,3)$, $(1,2)$, and $(1,3)$, respectively).

Figure~\ref{fig:orbit} shows a partial orbit 
of $\rowB$ for $P=[3] \times [3]$, depicting a generic $P$-array $f$
along with the associated arrays $\rowB f$, $\rowB^2 f$, and $\rowB^3 f$.
Consider the array $g=\recom f$ formed by recombining
the bottom negative fiber of $f$ (consisting of $a_1,a_2,a_3$),
the middle negative fiber of $\rowB f$ (consisting of $b_4,b_5,b_6$),
and the top negative fiber of $\rowB^2 f$ (consisting of $c_7,c_8,c_9$),
as shown in the left frame of Figure~\ref{fig:fibers}.
We claim that $\proB g$ coincides with the array formed by recombining
the bottom negative fiber of $\rowB f$ (consisting of $b_1,b_2,b_3$),
the middle negative fiber of $\rowB^2 f$ (consisting of $c_4,c_5,c_6$),
and the top negative fiber of $\rowB^3 f$ (consisting of $d_7,d_8,d_9$),
as shown in the right frame of Figure~\ref{fig:fibers}.
For example, consider what happens when we compute 
the middle (i.e., $(2,2)$) entry of $\proB g$, 
assuming that our claim applies to the previously-computed entries.
The left panel of Figure~\ref{fig:key} shows the array 
four-ninths of the way through the process of applying
rowmotion to $\rowB(f)$ to compute $\rowB^2(f)$;
it consists of the five entries $b_5$, $b_9$, $b_2$, $b_6$, $b_3$ from $\rowB(f)$
(seen in the upper right panel of Figure~\ref{fig:orbit})
and the four entries $c_7$, $c_4$, $c_8$, $c_1$ from $\rowB^2(f)$
(seen in the lower left panel of Figure~\ref{fig:orbit}).
The right panel of Figure~\ref{fig:key} shows the array 
four-ninths of the way through the process of applying
promotion to $g$ to compute $\pi_\cB(g)$;
it consists of the five entries $b_5$, $a_3$, $c_8$, $b_6$, $c_9$ from $g$
(seen in the left panel of Figure~\ref{fig:fibers})
and the four entries $b_1$, $c_4$, $b_2$, $d_7$ from $\pi_\cB(g)$
(seen in the right panel of Figure~\ref{fig:fibers}).
In both cases, when one toggles the middle entry
and replaces $b_5$ by $c_5$,
with $c_5$ expressed in terms of previously-computed entries,
the governing relation is $b_5 c_5 = (b_2 + b_6)(c_4 \psum c_8)$.
(Indeed, the proof of Theorem~\ref{thm:delta} was found
by thinking hard about why the four entries
that adjoin $b_5$ are the same in both panels,
and generalizing.)

\begin{defi}
Given $a,b \geq 1$ and $P = [a] \times [b]$,
and given a $P$-array $f: P \rightarrow \R^+$,
let
\begin{equation} \label{eq:deltaij}
(\recom f)(i,j) = (\rowB^{j-1}f)(i,j).
\end{equation}
\end{defi}

\noindent
That is, $\recom f$ is the $P$-array
whose $b$ negative fibers, read from bottom to top,
have the same values as the corresponding negative fibers
in $f$, $\rowB f$, $\rowB^2 f$, \dots, $\rowB^{b-1} f$.
The operation $\recom$ can also be expressed as a product of toggles.
Specifically, for $1 \leq j \leq b$ let $T_j$ be the operation
$\tau_{(1,j)} \circ \tau_{(2,j)} \circ 
\cdots \circ \tau_{(a-1,j)} \circ \tau_{(a,j)}$
that toggles the elements of the $j$th negative fiber
from left to right (or equivalently from top to bottom),
and for $1 \leq j \leq b$ let $U_j$ be the operation
$T_{b-j+1} \circ T_{b-j+2} \circ \cdots \circ T_{b-1} \circ T_{b}$
that toggles the top $j$ negative fibers from top to bottom;
then $\recom = U_1 \circ U_2 \circ \cdots \circ U_{b-2} \circ U_{b-1}$.
It is no accident that a similar product of toggles
can be found in Theorem 5.4 of~\cite{SW12}.
Note that this representation of recombination as a composition of involutions
makes it clear that $\recom$ is invertible.
For a concrete formula for the inverse map, see Lemma~\ref{lem:inverse}.

Now we come to the main result of this section,
a birational analogue of Striker and Williams' main result.

\begin{theo}
\label{thm:delta}
Fix $a,b \geq 1$, let $n=a+b$, and let $P = [a] \times [b]$.
Then for all $f: P \rightarrow \R^+$, 
\begin{equation} \label{eq:recom}
\recom \rowB f = \proB \recom f.
\end{equation}
\end{theo}

\begin{proof}
We will show that
\begin{equation}
\label{eq:restate}
(\recom \rowB f)(i,j) = (\proB \recom f)(i,j)
\end{equation}
for all $(i,j) \in [a] \times [b]$ by left-to-right induction 
(starting with $(a,1)$ and ending with $(1,b)$).
The reader may find it helpful to consult Figure~\ref{fig:local})
which shows the vicinity of a typical element $x=(i,j)$
in the poset $[a] \times [b]$ (away from the boundary).
\begin{figure}[ht] 
\[
\xymatrixrowsep{1.5pc}\xymatrixcolsep{0.1pc}\xymatrix{
(i\!+\!1,j)\sh &       & (i,j\!+\!1) \\
        & (i,j)\nh\sh &         \\
(i,j\!-\!1)\nh &       & (i\!-\!1,j)
}
\]
\caption{The local picture.}
\label{fig:local}
\end{figure}
In the boundary cases where one or more of the ordered pairs
$(i+1,j),(i,j+1),(i,j-1),(i-1,j)$ does not belong to $[a] \times [b]$,
the out-of-bounds ordered pair(s) may be ignored.

Note that $(i,j-1)$ and $(i+1,j)$ are to the left of $(i,j)$.
If we now assume that (\ref{eq:restate}) holds for $(i,j-1)$ 
(which we take to be vacuously true if $(i,j-1) \not\in [a] \times [b]$)
and for $(i+1,j)$
(which we take to be vacuously true if $(i+1,j) \not\in [a] \times [b]$),
then 
\begin{equation}
\label{eq:iha}
(\recom \rowB f)(i,j-1) = (\proB \recom f)(i,j-1)
\end{equation}
and 
\begin{equation}
\label{eq:ihb}
(\recom \rowB f)(i+1,j) = (\proB \recom f)(i+1,j).
\end{equation}

Note that $\rowB$ can be described directly via the recurrence
\begin{equation}
\label{eq:rowrecur}
(\rowB f)(x) \ = \ \frac{1}{f(x)} \ \bigsum \{f(y): \: y \in x^-\}
\ \bigpsum \{(\rowB f)(y): \: y \in x^+\}
\end{equation}
for all $x \in P$,
where in lieu of including $\hat{0}$ and $\hat{1}$
we interpret $\bigsum$ and $\bigpsum$ of the empty set 
as $\alpha$ and $\omega$ respectively.
We rewrite (\ref{eq:rowrecur}) for $x=(i,j)$ as
\begin{equation}
\label{eq:rowij}
\begin{split}
(\rowB f)(i,j) \ = \ \frac{1}{f(i,j)} 
\ \times \ \bigsum \{f(i-1,j),f(i,j-1)\} \\
\times \ \bigpsum \{(\rowB f)(i+1,j),(\rowB f)(i,j+1)\},
\end{split}
\end{equation}
where terms $f(\cdot,\cdot)$ and $(\rowB f)(\cdot,\cdot)$
are to be ignored if the arguments 
do not belong to $[a] \times [b]$.
Likewise, promotion can be described by the recurrence
\begin{equation}
\begin{split}
\label{eq:proij}
(\proB f)(i,j) \ = \ \frac{1}{f(i,j)} 
\ \times \ \bigsum \{f(i-1,j),(\proB f)(i,j-1)\} \\
\times \ \bigpsum \{(\proB f)(i+1,j),f(i,j+1)\}.
\end{split}
\end{equation}

From (\ref{eq:deltaij}) we obtain
$(\recom f)(i,j+1) = (\rowB^{j} f)(i,j+1)$
and
$(\recom f)(i,j-1) = (\rowB^{j-2} f)(i,j-1)$,
which we will use presently.

One the one hand we have
\begin{eqnarray*}
(\recom \rowB f)(i,j) 
& = & (\rowB^{j-1} \rowB f)(i,j) 
\ \mbox{(by (\ref{eq:deltaij}))} \\
& = & (\rowB \rowB^{j-1} f)(i,j) \\
& = & \frac{1}{(\rowB^{j-1} f)(i,j)} 
\ \times \ \bigsum \{(\rowB^{j-1} f)(i-1,j),(\rowB^{j-1} f)(i,j-1)\} \\
& & \ \times \ \bigpsum 
    \{(\rowB \rowB^{j-1} f)(i+1,j),(\rowB \rowB^{j-1} f)(i,j+1)\} 
\ \mbox{(by (\ref{eq:rowij}))} \\ 
& = & \frac{1}{(\rowB^{j-1} f)(i,j)} 
\ \times \ \bigsum \{(\rowB^{j-1} f)(i-1,j),(\rowB^{j-1} f)(i,j-1)\} \\
& & \ \times \ \bigpsum \{(\rowB^{j} f)(i+1,j),(\rowB^{j} f)(i,j+1)\} .
\end{eqnarray*}
On the other hand we have
\begin{eqnarray*}
(\proB \recom f)(i,j) 
& = & \frac{1}{(\recom f)(i,j)} \textbf{}
\ \times \ \bigsum \{(\recom f)(i-1,j),(\proB \recom f)(i,j-1)\} \\
& & \ \times \bigpsum \{(\proB \recom f)(i+1,j),(\recom f)(i,j+1)\} 
\ \mbox{(by (\ref{eq:proij}))} \\
& = & \frac{1}{(\rowB^{j-1} f)(i,j)} 
\ \times \ \bigsum \{ (\rowB^{j-1} f)(i-1,j), (\proB \recom f)(i,j-1)\} \\
& & \ \times \bigpsum \{(\proB \recom f)(i+1,j), (\rowB^{j} f)(i,j+1) \} 
\ \mbox{(by (\ref{eq:deltaij}))} \\
& = & \frac{1}{(\rowB^{j-1} f)(i,j)} 
\ \times \ \bigsum \{ (\rowB^{j-1} f)(i-1,j), (\recom \rowB f)(i,j-1)\} \\
& & \ \times \bigpsum \{(\recom \rowB f)(i+1,j), (\rowB^{j} f)(i,j+1) \} 
\ \mbox{(by (\ref{eq:iha}) and (\ref{eq:ihb}))} \\ 
& = & \frac{1}{(\rowB^{j-1} f)(i,j)} 
\ \times \ \bigsum \{ (\rowB^{j-1} f)(i-1,j), (\rowB^{j-2} \rowB f)(i,j-1)\} \\
& & \ \times \bigpsum \{(\rowB^{j-1} \rowB f)(i+1,j), (\rowB^{j} f)(i,j+1) \} 
\ (\mbox{by (\ref{eq:deltaij})}) \\
& = & \frac{1}{(\rowB^{j-1} f)(i,j)} 
\ \times \ \bigsum \{ (\rowB^{j-1} f)(i-1,j), (\rowB^{j-1} f)(i,j-1)\} \\
& & \ \times \bigpsum \{(\rowB^{j} f)(i+1,j), (\rowB^{j} f)(i,j+1) \} .
\end{eqnarray*}
Comparing the final expressions in the two equation blocks,
we conclude that $(\recom \rowB f)(i,j) = (\proB \recom f)(i,j)$,
which was to be proved.
\end{proof}

\begin{coro} \label{cor:repeat}
For all $k \geq 0$,
\begin{equation}
\label{eq:deltarepeat}
\recom \rowB^k f = \proB^k \recom f
\end{equation}
\end{coro}

\begin{proof}
Immediate.
\end{proof}

\begin{coro} \label{cor:order}
The map $\proB$ (for $P = [a] \times [b]$)
is of order $a+b$.
\end{coro} 

\begin{proof}
This follows from Grinberg-Roby periodicity,
Corollary~\ref{cor:repeat},
and the fact that $\recom$ is bijective.
\end{proof}

\begin{lemma} \label{lem:inverse}
The map $\decom$ defined by
\begin{equation}
(\decom g)(i,j) = (\proB^{n+1-j} g)(i,j)
\end{equation}
satisfies $\decom \recom f = f$
and hence coincides with $\recom^{-1}$.
\end{lemma}

\begin{proof}
\begin{eqnarray*}
(\decom \recom f)(i,j) 
& = & (\proB^{n+1-j} (\recom f))(i,j) \\
& = & (\recom \rowB^{n+1-j} f)(i,j) \\
& & \ \ \ \ \mbox{(by Corollary~\ref{cor:repeat} with $k = n+1-j$)} \\
& = & (\rowB^{j-1} \rowB^{n+1-j} f)(i,j) \\
& & \ \ \ \ \mbox{(by (\ref{eq:deltaij}) applied to $\rowB^{n+1-j} f$)} \\
& = & (\rowB^n f)(i,j) 
\end{eqnarray*}
which equals $f(i,j)$ by the periodicity of $\rowB$.
\end{proof}

\begin{theo}
Fix $a,b \geq 1$, let $n=a+b$, and let $P = [a] \times [b]$.
For any statistic $F$ on $(\R^+)^P$
of the form $Ff = \prod_{x \in P} f(x)^{a_x}$
with $a_x \in \Z$ for all $x \in P$,
$$\prod_{k=0}^{n-1} F(\rowB^k f) = \prod_{k=0}^{n-1} F(\proB^k \recom f).$$
In particular, if $F$ is homomesic with respect to promotion,
$F$ is homomesic with respect to rowmotion.
\end{theo}

\begin{proof}
Recall that $\rowB^n$ and $\proB^n$ are both the identity map.
It will suffice to consider $F$ of the form
$F_x f = f(x)^{a_x}$ for some particular $x \in P$, say $x=(i,j)$.
Theorem~\ref{thm:delta} implies $\rowB^k f = \recom^{-1} \proB^k \recom f$.
Then setting $g_j = \proB^{n+1-j} \recom f$ we have
\begin{eqnarray*}
F_x (\rowB^k f) 
& = & ((\rowB^k f)(x))^{a_x} \\
& = & ((\recom^{-1} \proB^k \recom f)(x))^{a_x} \\
& = & ((\proB^{n+1-j} \proB^k \recom f)(x))^{a_x} 
\ \mbox{(by Lemma~\ref{lem:inverse})} \\
& = & ((\proB^k \proB^{n+1-j} \recom f)(x))^{a_x} \\
& = & F_x (\proB^k g_j)
\end{eqnarray*}
so that
\begin{eqnarray*}
\prod_{k=0}^{n-1} F_x (\rowB^k f) & = & \prod_{k=0}^{n-1} F_x (\proB^k g_j) \\
& = & \prod_{k=0}^{n-1} F_x (\proB^{k+n+1-j} \recom f) \\
& = & \prod_{k=0}^{n-1} F_x (\proB^k \recom f)
\end{eqnarray*}
(by reindexing and appealing to periodicity).
Since this holds for all $x$, 
the claim follows by multiplication.
\end{proof}

\end{section}

\begin{section}{Tropicalization}
\label{sec:tropic}

Theorem~\ref{thm:prodhrefined} is nothing more than 
a complicated identity involving the operations
of multiplication, division, addition, and parallel addition.
As such, Theorem~\ref{thm:prodhrefined} can be tropicalized to yield
an identity involving the operations
of addition, subtraction, min, and max.
The resulting identity is Theorem~\ref{thm:sumhrefined}.
Here we provide details of the tropicalization process.

\begin{lemma}
\label{lem:tropic}
Suppose $E_1(t_1,\dots,t_r)$ and $E_2(t_1,\dots,t_r)$ 
are subtraction-free rational functions,
expressed in terms of the number 1 
and the operations $+$, $\psum$, $\times$, and $/$,
such that $E_1(t_1,\dots,t_r) = E_2(t_1,\dots,t_r)$ 
for all $t_1,\dots,t_r$ in $\R^+$.
Let $e_i(t_1,\dots,t_r)$ \textup{(}for $1 \leq i \leq 2$\textup{)}
be the result of replacing 
$1$, $+$, $\psum$, $\times$, and $/$
by $0$, $\min$, $\max$, $+$, and $-$, respectively.
Then $e_1(t_1,\dots,t_r) = e_2(t_1,\dots,t_r)$ 
for all $t_1,\dots,t_r$ in $\R$.
\end{lemma}

\begin{proof}
Write $e_i(t_1,\dots,t_m) = -\lim_{N \rightarrow \infty}
\frac{1}{N} \log E_i(e^{-Nt_1},\dots,e^{-Nt_m})$.\footnote{We are indebted 
to Colin McQuillan and Will Sawin for clarifying this point; see~\cite{MO13}.}
\end{proof}

It is worth observing that the converse of the lemma is false.
That is, there are tropical identities like
$\max(t,t)=t$ whose ``detropicalization'' is false.
Thus, while Theorem~\ref{thm:sumhrefined} 
is a consequence of Theorem~\ref{thm:prodhrefined},
Theorem~\ref{thm:prodhrefined} cannot be derived 
as a consequence of Theorem~\ref{thm:sumhrefined},
at least not using any methods we are aware of.

At the same time, we should mention that in a certain sense,
Theorem~\ref{thm:sumhrefined} can be proved 
without relying on Theorem~\ref{thm:prodhrefined}.
Specifically, one can tropicalize each of the steps 
in the proof of Theorem~\ref{thm:prodhrefined},
using (for instance) the identity
\begin{equation}
\min(x,y)+\max(x,y) = x+y
\end{equation}
in place of the identity (\ref{eq:duality}).

\begin{proof}[Proof of Theorem~\ref{thm:sumhrefined}]
This is an immediate consequence of Theorem~\ref{thm:prodhrefined}
via tropicalization.
\end{proof}

\begin{proof}[Proof of Theorem~\ref{thm:sumrefined}]
Putting $\alpha=1$ in Theorem~\ref{thm:prodrefined},
we get
$$\prod_{k=0}^{n-1} |\proB^k (v)|_\ell = 
\left\{ \begin{array}{ll}
\omega^{a \ell} & \mbox{if $1 \leq \ell \leq b$}, \\
\omega^{b(n - \ell)} & \mbox{if $b \leq \ell \leq n$}.
\end{array} \right.$$
Tropicalizing these result gives Theorem~\ref{thm:sumrefined}.
\end{proof}

\begin{proof}[Proof of Theorem~\ref{thm:cardrefined}]
Putting $\omega=1$ in Theorem~\ref{thm:prodrefined},
we get
$$\prod_{k=0}^{n-1} |\proB^k (v)|_\ell = 
\left\{ \begin{array}{ll}
\alpha^{b \ell} & \mbox{if $\ell \leq a$}, \\
\alpha^{a(n - \ell)} & \mbox{if $a \leq \ell n$}.
\end{array} \right.$$
Tropicalizing this result gives a homomesy result
for the promotion on the reverse order polytope
(the set of order-reversing maps 
from $[a] \times [b]$ to $[0,1]$).
Specializing to the vertices of this polytope
gives the desired homomesy for order ideals.
\end{proof}

\end{section}

\begin{section}{Other homomesies} 
\label{sec:opposites}

Recall that $P_i$ ($1 \leq i \leq n-1$)
denotes the $i$th file of $P = [a] \times [b]$.
We have shown that for $P = [a] \times [b]$,
the functions $$p_i: f \mapsto \sum_{\mbox{$x \in P_i$}} f(x)$$
(and their analogues in the birational setting)
are homomesic under the action of rowmotion and promotion.
These are not, however, the only combinations of
the local evaluation operations $f \mapsto f(x)$ ($x \in P$)
that exhibit homomesy.
Specifically, we now show that for all $a,b$,
the functions $f \mapsto f(x) + f(x')$ are homomesic
under rowmotion and promotion,
where $x=(i,j)$ and $x'=(i',j')$ are opposite elements of $P$,
that is, $i+i'=a+1$ and $j+j'=b+1$.
We will prove the birational version,
since the PL version is an easy consequence.
In the birational setting, the statistic is $f \mapsto f(x) f(x')$.

\begin{theo}
Fix $a,b \geq 1$, let $n=a+b$, and let $P = [a] \times [b]$.
Fix $(i,j),(i',j')$ in $P$ satisfying $i+i'=a+1$ and $j+j'=b+1$.
For $f:P \rightarrow \R^+$ let $F(f) = f(i,j) f(i',j')$.
Then $F$ is multiplicatively $\alpha \omega$-mesic under $\rowB$.
\end{theo}

\begin{proof}
Theorem 32 in~\cite{GR15} says that $f(i',j') f'(i,j) = \alpha \omega$
with $f' = \rowB^{i+j-1} f$, so 
\begin{eqnarray*}
(\alpha \omega)^n
& = & \prod_{k=1}^n (\rowB^k f)(i',j') (\rowB^k f')(i,j) \\
& = & \prod_{k=1}^n (\rowB^k f)(i',j') \ \prod_{k=1}^n (\rowB^k f')(i,j) \\
& = & \prod_{k=1}^n (\rowB^k f)(i',j') \ \prod_{k=1}^n (\rowB^k f)(i,j) \\
&   & \mbox{(since $\rowB$ has period $n$)} \\
& = & \prod_{k=1}^n F (\rowB^k f)
\end{eqnarray*}
as claimed.
\end{proof}

\end{section}

\longthanks{The authors are grateful to Arkady Berenstein, Darij Grinberg, 
Michael Joseph, Tom Roby, Richard Stanley, and Jessica Striker 
for helpful conversations and detailed comments on the manuscript.
The two anonymous referees also had many good suggestions.}

\nocite{*}
\bibliographystyle{amsplain-ac}
\bibliography{ep-may2020}
\end{document}